\documentclass[a4paper,12pt]{article}
\usepackage{latexsym}
\usepackage{amsmath}
\usepackage{amssymb}
\usepackage{enumerate}
\usepackage{bm}

\usepackage{theorem}
\newtheorem{theorem}{Theorem}[section]
\newtheorem{proposition}[theorem]{Proposition}
\newtheorem{lemma}[theorem]{Lemma}

\theorembodyfont{\rmfamily}
\newtheorem{proof}{\textmd{\textit{Proof.}}}

\newtheorem{remark}[theorem]{Remark}
\newtheorem{example}[theorem]{Example}
\newtheorem{definition}[theorem]{Definition}
\newtheorem{ack}{Acknowledgments}

\makeatletter

\@addtoreset{equation}{section}
\makeatother

\newcommand{\qedd}{\hfill \Box}
\newcommand{\ve}{\varepsilon}
\newcommand{\del}{\partial}
\newcommand{\lra}{\longrightarrow}

\newcommand{\N}{\ensuremath{\mathbb{N}}}

\newcommand{\R}{\ensuremath{\mathbb{R}}}

\newcommand{\bD}{\ensuremath{\mathbf{D}}}
\newcommand{\bH}{\ensuremath{\mathbf{H}}}
\newcommand{\bK}{\ensuremath{\mathbf{K}}}
\newcommand{\bL}{\ensuremath{\mathbf{L}}}
\newcommand{\bR}{\ensuremath{\mathbf{R}}}
\newcommand{\bs}{\ensuremath{\mathbf{s}}}
\newcommand{\cC}{\ensuremath{\mathcal{C}}}
\newcommand{\cD}{\ensuremath{\mathcal{D}}}
\newcommand{\cE}{\ensuremath{\mathcal{E}}}
\newcommand{\cG}{\ensuremath{\mathcal{G}}}
\newcommand{\cH}{\ensuremath{\mathcal{H}}}
\newcommand{\cJ}{\ensuremath{\mathcal{J}}}
\newcommand{\cP}{\ensuremath{\mathcal{P}}}
\newcommand{\cT}{\ensuremath{\mathcal{T}}}
\newcommand{\cV}{\ensuremath{\mathcal{V}}}
\newcommand{\fm}{\ensuremath{\mathfrak{m}}}
\newcommand{\sT}{\ensuremath{\mathsf{T}}}
\newcommand{\sU}{\ensuremath{\mathsf{U}}}

\def\diam{\mathop{\mathrm{diam}}\nolimits}
\def\vol{\mathop{\mathrm{vol}}\nolimits}
\def\div{\mathop{\mathrm{div}}\nolimits}
\def\Hess{\mathop{\mathrm{Hess}}\nolimits}

\def\loc{\mathop{\mathrm{loc}}\nolimits}

\def\Ric{\mathop{\mathrm{Ric}}\nolimits}

\def\id{\mathop{\mathrm{id}}\nolimits}
\def\Ent{\mathop{\mathrm{Ent}}\nolimits}
\def\trace{\mathop{\mathrm{trace}}\nolimits}

\def\CD{\mathop{\mathrm{CD}}\nolimits}

\def\HS{\mathop{\mathrm{HS}}\nolimits}

\newcommand{\Nabla}{\bm{\nabla}}
\newcommand{\Lap}{\bm{\Delta}}
\newcommand{\btau}{\bm{\tau}}

\newcommand{\rev}[1]{\overleftarrow{#1}}
\newcommand{\wt}[1]{\widetilde{#1}}
\newcommand{\ol}[1]{\overline{#1}}
\newcommand{\ora}[1]{\overrightarrow{#1}}

\title{On the curvature and heat flow\\ on Hamiltonian systems}
\author{Shin-ichi OHTA\thanks{Department of Mathematics, Kyoto University,
Kyoto 606-8502, Japan ({\sf sohta@math.kyoto-u.ac.jp});
Supported in part by the Grant-in-Aid for Young Scientists (B) 23740048.}}
\date{}
\begin{document}

\maketitle

\begin{abstract}
We develop the differential geometric and geometric analytic studies of Hamiltonian systems.
Key ingredients are the curvature operator,
the weighted Laplacian, and the associated Riccati equation.
We prove appropriate generalizations of
the Bochner--Weitzenb\"ock formula and Laplacian comparison theorem,
and study the heat flow.
\bigskip

\end{abstract}

\tableofcontents

\section{Introduction}\label{sc:intro}%%%%%%%%%%
%%%%%%%%%%%%%%%%%%%%%%

The aim of this article is to apply the recently developed technique in Finsler geometry
to the study of Hamiltonian systems.
A Finsler manifold carries a (Minkowski) norm on each tangent space.
Although Finsler manifolds form a much wider class than Riemannian manifolds,
the notion of curvature makes sense and we can consider various comparison theorems
similarly to the Riemannian case (see, e.g., \cite{Shvol}).
Especially, the weighted Ricci curvature introduced by the author~\cite{Oint}
has fruitful applications including the curvature-dimension condition (\cite{Oint}),
Laplacian comparison theorem for the natural nonlinear Laplacian (\cite{OShf}),
Bochner--Weitzenb\"ock formula and gradient estimates (\cite{OSbw}),
and generalizations of the Cheeger--Gromoll splitting theorem (\cite{Ospl}).
To be precise, the weighted Ricci curvature is defined for a pair consisting of a Finsler manifold
and a measure on it, and our Laplacian depends on the choice of the measure
(see Subsections~\ref{ssc:Fins}, \ref{ssc:Lap} for details).

Then, it is natural to expect that the theory of curvature can be applied beyond Finsler manifolds,
and a class of manifolds $M$ endowed with Lagragians $\bL$ or,
equivalently, Hamiltonians $\bH$ is a natural choice.
In fact, on the one hand, we know that Agrachev and Gamkrelidze~\cite{AG} (see also \cite{Agr})
have developed the theory of curvature operator for Hamiltonian systems in connection
with optimal control theory (see \cite{Agr1} for a dynamical application).
On the other hand,
optimal transport theory (which is related to the curvature-dimension condition)
for Lagrangian cost functions has been well investigated (\cite{BeBu}, \cite{FF}, \cite{Vi2}).
Furthermore, Lee~\cite{Le2} recently showed a Riccati equation
(see also \cite{ALge}, \cite{ALbl}, \cite{LLZ} for the sub-Riemannian case)
as well as convexity estimates for entropy functionals along smooth optimal transports
for general (time-dependent) Hamiltonian systems, by means of the curvature operator.
His unified approach recovers both the curvature-dimension condition
($\CD(K,\infty)$ and $\CD(0,N)$ to be precise) for Riemannian or Finsler manifolds
and various monotonicity formulas along flows in Riemannian metrics
related to the Ricci flow.

Our Hamiltonian will always be time-independent and non-negative.
Compared with the Finsler situation, the lack of the homogeneity causes many difficulties,
for example,
we need to take care of the difference between the Lagrangian and the Hamiltonian
(they coincide as functions via the Legendre transform in the Finsler case).
Nevertheless, by combining the Riccati equation and the technique in the Finsler case,
we prove the Bochner--Weitzenb\"ock formula (Theorem~\ref{th:BW})
and Laplacian comparison theorem (Theorem~\ref{th:Lcomp}).
We also obtain functional inequalities from the convexity of the relative entropy (Theorem~\ref{th:funct}).
In these results, we use the weighted Ricci curvature derived from the curvature operator
as well as the Laplacian $\Lap^{\bH}_{\fm}$ induced from the Hamiltonian $\bH$
and the reference measure $\fm$ on $M$.
In general, this Laplacian is not only nonlinear
($\Lap^{\bH}_{\fm}(f+g) \neq \Lap^{\bH}_{\fm}f +\Lap^{\bH}_{\fm} g$)
but also non-homogeneous ($\Lap^{\bH}_{\fm} (cf) \neq c\Lap^{\bH}_{\fm} f$ for $c \in \R$).

We also study the evolution equation $\del_t u=\Lap^{\bH}_{\fm}u$
which can be thought of as the `heat equation' in our context.
In various settings including Finsler manifolds (\cite{OShf}),
the heat flow is regarded as gradient flow in the following two ways:
\begin{enumerate}[(I)]
\item the gradient flow of the Dirichlet energy in the $L^2$-space;
\item the gradient flow of the relative entropy in the $L^2$-Wasserstein space.
\end{enumerate}
The identification of these two strategies was recently established
for general metric measure spaces satisfying the curvature-dimension condition (\cite{AGShf}).
Although a Hamiltonian does not induce a distance function,
we verify the analogue of the former approach (I) in Theorem~\ref{th:Lgf}.
Indeed, since our energy form $\cE(u)=\int_M \bH(du) \,d\fm$
is a convex functional on the $L^2$-space,
the classical theory of Br\'ezis et al applies.
Interestingly, however, the latter strategy (II) leads to the different equation
$\del_t u=-\div_{\fm}(u\Nabla[-\log u])$ (Theorem~\ref{th:Wgf}).
This shows that the identification of (I) and (II) is essentially due to the homogeneity of $\bH$.

We explicitly calculate the curvature and Laplacian in the special class
of Hamiltonians given by convex deformations of Finsler Hamiltonians
(that is, $\bH=h \circ F^*$, where $h$ is a convex function and
$F^*$ is the dual of a Finsler structure $F:TM \lra \R$; see Subsection~\ref{ssc:Fconv}).
What is of particular interest is the $p$-homogeneous deformations
($h(t)=t^p/p$), which derives the Finsler analogue of $p$-Laplacians.
However, we do not focus on this specific case
because the aim of this work is to present the general framework.
In the same spirit, we sometimes discuss under seemingly superfluous assumptions
for the sake of technical simplicity.

\begin{ack}
I am grateful to Professor Paul W.~Y.~Lee for his helpful comments on a preliminary version of the paper.
\end{ack}

\section{Preliminaries}\label{sc:prel}%%%%%%%%%%
%%%%%%%%%%%%%%%%%%%%%%

Throughout the article,
let $M$ be a connected $\cC^{\infty}$-manifold of dimension $n \ge 2$ without boundary.
For a local coordinate system $(x^i)_{i=1}^n$ on an open set $U \subset M$,
we will always consider the (fiber-wise linear) coordinates
$(x^i;v^j)_{i,j=1}^n$ of the tangent bundle $TU$
and $(x^i;\alpha_j)_{i,j=1}^n$ of the co-tangent bundle $T^*U$ given by
\[ v=\sum_{j=1}^n v^j \del_{x^j}=\sum_{j=1}^n v^j \frac{\del}{\del x^j},
 \qquad \alpha=\sum_{j=1}^n \alpha_j dx^j, \]
respectively.
We will use the usual abbreviations such as
\[ \bL_{v^i}:=\frac{\del \bL}{\del v^i}, \qquad
 \bH_{x^i \alpha_j}:=\frac{\del^2 \bH}{\del \alpha_j \del x^i} \]
for brevity, but only for the Lagrangian $\bL$ on $TM$ and the Hamiltonian $\bH$ on $T^*M$.

\subsection{Lagrangians}\label{ssc:Lag}%%%%%%%%%%%%%%

We consider only time-independent (autonomous) Lagrangians for simplicity.
General time-dependent Lagrangians can be treated similarly to a great extent
(see \cite{AG}, \cite{Agr}, \cite{Le2}).

\begin{definition}[Lagrangians]\label{df:Lag}
A \emph{Lagrangian} on $M$ will mean a non-negative function
$\bL:TM \lra [0,\infty)$ satisfying the following conditions:
\begin{enumerate}[(1)]
\item $\bL \in \cC^1(TM) \cap \cC^{\infty}(TM \setminus 0_{TM})$ and $\bL \equiv 0$ on $0_{TM}$,
where $0_{TM} \subset TM$ denotes the zero section;
\item (\emph{Super-linearity})
There are some complete Riemannian metric $g$ of $M$ and positive constant $C>0$ such that
$\bL(v) \ge |v|_g -C$ for all $v \in TM$;
\item (\emph{Strong convexity})
For any $x \in M$, $\bL$ has the positive-definite Hessian at every $v \in T_xM \setminus \{0\}$
(with respect to an arbitrary linear coordinate of $T_xM$).
\end{enumerate}
\end{definition}

Note that the strong convexity implies $\bL>0$ on $TM \setminus 0_{TM}$
and that the super-linearity (2) follows from (1) and (3) if $M$ is compact.
The reason why only the $\cC^1$-regularity is assumed on $0_{TM}$
is to include non-Riemannian Finsler metrics (see Remark~\ref{rm:Lag}(a)).
Then the non-negativity $\bL \ge 0$ and $\bL|_{0_{TM}} \equiv 0$
are imposed to ensure that an action-minimizing curve
$\eta:[0,T] \lra M$ with $\dot{\eta}(0) \neq 0$ always enjoys
$\dot{\eta} \neq 0$ on whole $[0,T]$ (due to the conservation of the Hamiltonian),
thus $\dot{\eta}$ lives in $TM \setminus 0_{TM}$ where $\bL$ is $\cC^{\infty}$.
These additional conditions ($\bL \ge 0$, $\bL|_{0_{TM}} \equiv 0$)
can be removed if $\bL \in \cC^2(TM)$ (as in \cite{Le2}).

A $\cC^1$-curve $\eta:[0,T] \lra M$ which is a critical point of the \emph{action}
$\int_0^T \bL(\dot{\eta}) \,dt$ (among variations fixing the endpoints)
is called an \emph{action-minimizing curve} and satisfies the \emph{Euler--Lagrange equation}
\begin{equation}\label{eq:EL}
\bL_{x^i}(\dot{\eta}) -\sum_{j=1}^n \{ \bL_{v^i x^j}(\dot{\eta}) \dot{\eta}^j
 +\bL_{v^i v^j}(\dot{\eta}) \ddot{\eta}^j \}=0
 \qquad \text{for all}\ i.
\end{equation}
To be precise, $\eta$ is either a constant curve or a $\cC^{\infty}$-curve with $\dot{\eta} \neq 0$.
Thus $\eta$ is always $\cC^{\infty}$ and \eqref{eq:EL} makes sense
(note that $\bL_{x^i}(0)=0$ since $\bL|_{0_{TM}} \equiv 0$).
Given any $v \in TM$, there is a unique action-minimizing curve
$\eta_v:(-\ve,\ve) \lra M$ with $\dot{\eta}_v(0)=v$ for sufficiently small $\ve>0$.
We say that $(M,\bL)$ is \emph{forward complete} (resp.\ \emph{complete})
if $\eta_v$ is extended to the action-minimizing curve
$\eta_v:[0,\infty) \lra M$ (resp.\ $\eta_v:\R \lra M$) for all $v \in TM$.

Let us summarize some fundamental remarks on the difference from
the Riemannian or Finsler case, caused by the non-homogeneity of $\bL$
(see also Remark~\ref{rm:Legt} below).

\begin{remark}\label{rm:Lag}
(a) If $\bL(cv)=c^2 \bL(v)$ for all $v \in TM$ and $c>0$,
then $F(v):=\sqrt{2\bL(v)}$ gives a Finsler metric
(see Subsections~\ref{ssc:Fins}, \ref{ssc:FinsH} below).
Moreover, in this case,
$F$ comes from a Riemannian metric if and only if $F^2 \in \cC^2(TM)$
(see \cite[Proposition~2.2]{Shvol}).

(b) An action-minimizing curve $\eta$ does not necessarily have a constant speed
(i.e., $\bL(\dot{\eta})$ may not be constant).
This is one of the reasons why the Hamiltonian (which is constant along the Hamiltonian flow)
fits better to our consideration.

(c) The strong convexity does not imply the uniform (strict) convexity of $\bL$
even in a single tangent space $T_xM$.
For instance, for $f(t)=|t|^p$ on $\R$ with $p \in (1,\infty)$,
we have $\lim_{t \to \infty}f''(t)=0$ if $p<2$ and $\lim_{t \downarrow 0}f''(0)=0$ if $p>2$.
This is one of the major differences from the Finsler setting,
in which the uniform convexity and smoothness
are used in various analytic and geometric estimates (see \cite[Sections~2, 3]{OShf}, \cite{Ouni}).
\end{remark}

\subsection{Hamiltonians}\label{ssc:Ham}%%%%%%%%%%%%

Let $\bL$ be a Lagrangian as in Definition~\ref{df:Lag}.
The associated \emph{Hamiltonian} is given by
\[ \bH(\alpha):=\sup_{v \in T_xM} \{ \alpha(v)-\bL(v) \}, \quad \alpha \in T_x^*M. \]
Choosing $v=0$ ensures that $\bH \ge 0$ as well as $\bH|_{0_{T^*M}} \equiv 0$.
Moreover,  $\bH>0$ on $T^*M \setminus 0_{T^*M}$ and
$\bH \in \cC^1(T^*M) \cap \cC^{\infty}(T^*M \setminus 0_{T^*M})$.
Given $\alpha \in T_x^*M$,
by virtue of the strong convexity of $\bL$, we can find a unique vector
$v \in T_xM$ satisfying $\bH(\alpha)=\alpha(v)-\bL(v)$.
Such a vector is denoted by $\btau^*(\alpha)$ and
called the \emph{Legendre transform} of $\alpha$,
and the map $\btau^*:T^*M \lra TM$ is explicitly written as
\begin{equation}\label{eq:btau*}
\btau^*(\alpha)=\sum_{i=1}^n \bH_{\alpha_i}(\alpha) \del_{x^i}.
\end{equation}
The inverse map $\btau:TM \lra T^*M$ of $\btau^*$ is similarly given by
\[ \btau(v) :=\sum_{i=1}^n \bL_{v^i}(v) dx^i. \]
Note that, for each $\alpha \neq 0$, $(\bH_{\alpha_i \alpha_j}(\alpha))_{i,j=1}^n$
is the inverse matrix of $(\bL_{v^i v^j}(\btau^*(\alpha)))_{i,j=1}^n$
so that it is positive-definite.

\begin{remark}\label{rm:Legt}
(a) The transform $\btau$ (or $\btau^*$) is a linear operator if and only if
$\bL$ comes from a Riemannian metric.
Precisely, $\btau(cv)=c\btau(v)$ holds for all $c>0$ and $v \in TM$
if and only if $\bL$ comes from a Finsler structure,
and $\btau(v+w)=\btau(v)+\btau(w)$ holds only in the Riemannian case.
Moreover, in the Finsler case, $\btau|_{T_xM}$ is differentiable at the origin
if and only if $F|_{T_xM}$ comes from an inner product.

(b) In the Finsler setting, $\bH(\btau(v))=\bL(v)$ holds and the better estimate
$\alpha(v) \le 2\sqrt{\bL(v) \bH(\alpha)}$ is available.
\end{remark}

The \emph{Hamiltonian vector field} on $T^*M$ is defined by
\[ \ora{\bH}(\alpha) :=\sum_{i=1}^n \{ \bH_{\alpha_i}(\alpha) \del_{x^i}
 -\bH_{x^i}(\alpha) \del_{\alpha_i} \} \in T_{\alpha}(T^*M). \]
Denote by $(\Phi_t(\alpha))_{t \in (-\ve,\ve)} \subset T^*M$ the corresponding \emph{Hamiltonian flow},
i.e., $\Phi_0(\alpha)=\alpha$ and $\dot{\Phi}_t(\alpha)=\ora{\bH}(\Phi_t(\alpha))$.
The Hamiltonian is preserved along $\Phi_t$,
namely $\bH(\Phi_t(\alpha))=\bH(\alpha)$ for all $t$.

A $\cC^{\infty}$-curve $\eta$ solves the Euler--Lagrange equation \eqref{eq:EL}
if and only if $\btau(\dot{\eta})$ verifies $\Phi_t(\btau(\dot{\eta}(0)))=\btau(\dot{\eta}(t))$.
That is to say, the Hamiltonian flow coincides with the flow on $TM$
generated from the Euler--Lagrange equation via the Legendre transform.
Thus the forward completeness (resp.\ completeness) of $\bL$
is equivalent to that of $\Phi_t$, namely the existence of $(\Phi_t)_{t \ge 0}$
(resp.\ $(\Phi_t)_{t \in \R}$) on whole $T^*M$.

Let us denote by $\omega=\sum_{i=1}^n d\alpha_i \wedge dx^i$
the canonical \emph{symplectic form} on $T^*M$.
Recall that
\begin{equation}\label{eq:omega}
\omega\big( d\Phi_t(v),d\Phi_t(w) \big) =\omega(v,w)
 \qquad \text{for all}\ v,w \in T_{\alpha}(T^*M),\ t \in (-\ve,\ve).
\end{equation}

\subsection{Finsler manifolds}\label{ssc:Fins}%%%%%%%%%%%%%%%

This subsection is devoted to a concise review on the special class of Finsler manifolds,
where we have a clearer understanding of curvature and heat flow.
We refer to \cite{BCS} and \cite{Shlec} for basics of Finsler geometry.

\begin{definition}[Finsler structures]\label{df:Fins}
We say that a function $F:TM \lra [0,\infty)$ is a \emph{$\cC^{\infty}$-Finsler structure}
of $M$ if it satisfies the following:
\begin{enumerate}[(1)]
\item $F \in \cC^{\infty}(TM \setminus 0_{TM})$;
\item (\emph{Positive $1$-homogeneity}) $F(cv)=cF(v)$ for all $v \in TM$ and $c>0$;
\item (\emph{Strong convexity}) For any $v \in TM \setminus 0_{TM}$,
the $(n \times n)$-symmetric matrix
\begin{equation}\label{eq:g_ij}
g_{ij}(v) :=\frac{1}{2} \frac{\del^2 (F^2)}{\del v^i \del v^j}(v)
\end{equation}
is positive-definite.
\end{enumerate}
\end{definition}

We do not assume the \emph{absolute homogeneity} $F(-v)=F(v)$ in general.
Thanks to the positive homogeneity, the Lagrangian $\bL^F:=F^2/2$
and the corresponding Hamiltonian $\bH^F$ for instance satisfy
\[ 2\bL^F(v)=\sum_{i,j=1}^n g_{ij}(v) v^i v^j, \qquad
 \bH^F(\alpha) =\frac{1}{2}F^*(\alpha)^2 =\bL^F\big( \btau^*(\alpha) \big) \]
for all $v \in T_xM$ and $\alpha \in T_x^*M$,
where $F^*|_{T_x^*M}$ is the dual (Minkowski) norm of $F|_{T_xM}$ given by
\[ F^*(\alpha) :=\sup_{v \in T_xM,\, F(v) \le 1} \alpha(v). \]
Euler's theorem on homogeneous functions is a key tool (see \cite[Theorem~1.2.1]{BCS}).
In the context of general Lagrangians, the lack of this basic tool
causes many differences from the Finsler setting.

Similarly to the Riemannian case, one can regard the Euler--Lagrange equation
as the \emph{geodesic equation} with respect to the distance structure induced from $F$.
Precisely, the \emph{distance function} $d_F:M \times M \lra [0,\infty)$ associated with $F$
is naturally introduced by
\[ d_F(x,y):=\inf_{\eta} \int_0^1 F(\dot{\eta}) \,dt, \]
where the infimum is taken over all piece-wise $\cC^1$-curves
$\eta:[0,1] \lra M$ with $\eta(0)=x$ and $\eta(1)=y$.
Note that $d_F$ can be nonsymmetric (i.e., $d_F(y,x) \neq d_F(x,y)$)
since $F$ is only positively homogeneous.
Then a curve $\eta$ solves the Euler--Lagrange equation if and only if
it is a \emph{geodesic} in the sense that it is locally $d_F$-minimizing
and of constant speed (i.e., $F(\dot{\eta})$ is constant).

In the Finsler world, corresponding to the sectional curvature of Riemannian manifolds
is the \emph{flag curvature} $\bK(v,w)$ for linearly independent vectors $v,w \in T_xM$.
We remark that, different from the Riemannian case,
$\bK(v,w)$ depends not only on the plane $v \wedge w$ spanned by $v$ and $w$
(the \emph{flag}), but also on the choice of $v$ in it (the \emph{flagpole}).

We know a useful interpretation of the flag curvature, due to Shen (\cite[\S 6.2]{Shlec}),
as follows.
Fix $v \in T_xM \setminus\{0\}$ and extend it to a $\cC^{\infty}$-vector field $V$
on a neighborhood $U$ of $x$ (i.e., $V(x)=v$) in such a way that all integral curves of $V$ are geodesic
(this is always possible, whereas the choice of $V$ is not unique).
By the strong convexity, $V$ induces the Riemannian structure
$g_V$ on $U$ via \eqref{eq:g_ij} as
\begin{equation}\label{eq:g_V}
g_V\left( \sum_{i=1}^n a_i \del_{x^i}, \sum_{j=1}^n b_j \del_{x^j} \right)
 :=\sum_{i,j=1}^n a_i b_j g_{ij}(V).
\end{equation}
Then, for $w \in T_xM$ which is not co-linear with $v$,
the sectional curvature of $v \wedge w$ with respect to $g_V$ coincides with $\bK(v,w)$
(independent of the choice of $V$).
This remarkable fact shows the usefulness of $V$ and $g_V$ as above
in the study of Finsler manifolds from the Riemannian geometric viewpoint.

For a unit vector $v \in T_xM \cap F^{-1}(1)$,
the \emph{Ricci curvature} $\Ric^F(v)$ is defined as the trace of $\bK(v,\cdot)$
with respect to $g_v$ (defined similarly to \eqref{eq:g_V}).
Thus $\Ric^F(v)$ coincides with the Ricci curvature of $v$ with respect to
the Riemannian structure $g_V$ as in the previous paragraph.
We also set $\Ric^F(cv):=c^2 \Ric^F(v)$ for $c \ge 0$.

Now we fix a positive $\cC^{\infty}$-measure $\fm$ on $M$ and modify $\Ric^F$
into the \emph{weighted Ricci curvature} $\Ric^F_N$ which was introduced in \cite{Oint}
inspired by the theory of weighted Riemannian manifolds.

\begin{definition}[Weighted Ricci curvature]\label{df:FRic}
Given $v \in T_xM \setminus \{0\}$, extend it to a $\cC^{\infty}$-vector field $V$
on a neighborhood $U$ of $x$ such that all integral curves of $V$ are geodesic.
Using the volume measure $\vol_V$ of $g_V$, we decompose $\fm$
as $\fm=e^{-\psi} \vol_V$ on $U$.
Let $\eta$ be the geodesic with $\dot{\eta}(0)=v$.
Then we define, for $N \in (n,\infty)$,
\[ \Ric^F_N(v):=\Ric^F(v) +(\psi \circ \eta)''(0) -\frac{(\psi \circ \eta)'(0)^2}{N-n}. \]
As the limits, define
\begin{align*}
\Ric^F_{\infty}(v) &:=\Ric^F(v) +(\psi \circ \eta)''(0), \\
\Ric^F_n(v) &:=\left\{
 \begin{array}{ll} \Ric^F(v)+(\psi \circ \eta)''(0) & \text{if}\ (\psi \circ \eta)'(0)=0, \\
 -\infty & \text{if}\ (\psi \circ \eta)'(0) \neq 0. \end{array} \right.
\end{align*}
We also set $\Ric^F_N:=0$ on $0_{TM}$ for all $N \in [n,\infty]$.
\end{definition}

It is easily seen that $\Ric^F_N(v)$ is well-defined (independent of the choice of $V$)
and $\Ric^F_N(cv)=c^2 \Ric^F_N(v)$ for all $c>0$.
In the Riemannian case, $\Ric^g_{\infty}$ is the famous \emph{Bakry--\'Emery tensor}
and $\Ric^g_N$ with $N<\infty$ was introduced by Qian (see \cite{BE}, \cite{Qi}).

\begin{remark}\label{rm:wRic}
On a Riemannian manifold $(M,g)$, if we choose the volume measure as the reference measure $\fm$,
then $\Ric^g_N=\Ric^g$ for all $N \in [n,\infty]$.
On a general Finsler manifold, however, there does not necessarily exist
a measure $\fm$ with $\Ric^F_n>-\infty$ (see \cite{ORand}).
This means that there is no nice reference measure in general,
so that it is natural to start from an arbitrary measure.
\end{remark}

It was demonstrated in \cite{Oint} that $(M,d_F,\fm)$ satisfies
Lott, Sturm and Villani's \emph{curvature-dimension condition} $\CD(K,N)$
if and only if $\Ric^F_N(v) \ge KF(v)^2$ for all $v \in TM$.
Roughly speaking, in $\CD(K,N)$, $K$ acts as the lower bound of the Ricci curvature
and $N$ is regarded as the upper bound of the dimension.
This characterization extends the Riemannian one
(by \cite{CMS}, \cite{vRS}, \cite{StI}, \cite{StII}, \cite{LV1}, \cite{LV2}),
and has many analytic and geometric applications
via the general theory of curvature-dimension condition
(developed in \cite{StI}, \cite{StII}, \cite{LV1}, \cite{LV2}, \cite{Vi2}).
Moreover, the Laplacian and heat flow on $(M,F,\fm)$
(both are nonlinear except for the Riemannian case) were studied in
\cite{OShf}, \cite{OSnc} and \cite{OSbw}, where we have shown
the Bochner--Weitzenb\"ock formula and the Bakry--\'Emery and Li--Yau gradient estimates among others.
See also \cite{Ospl} for a further application to
generalizations of the Cheeger--Gromoll splitting theorem.

\section{Curvature for Hamiltonians}\label{sc:curv}%%%%%%%%%%
%%%%%%%%%%%%%%%%%%%%%%

In \cite{AG}, Agrachev and Gamkrelidze introduced the curvature operator for Hamiltonian systems.
We recall their definition along the line of \cite{Le2} (see also \cite{Le1}).
Although our Lagrangian is assumed to be time-independent and non-negative,
the original definition is concerned with general time-dependent Lagrangians.

\begin{remark}\label{rm:AD}
According to \cite{AD},
the construction of curvature in \cite{AG} is equivalent to those in the independent works \cite{Gr} and \cite{Fo}
(see also Acknowledgment in the preprint version {\sf arXiv:1205.1442v6} of \cite{Le2}).
We refer to \cite{AD} and the references therein for details and some other related works.
\end{remark}

Fix $\alpha \in T_x^*M \setminus \{0\}$ and put $\bm{\alpha}(t):=\Phi_t(\alpha)$
for $t \in (-\ve,\ve)$ throughout this section.
For each $t \in (-\ve,\ve)$, let us consider
\begin{align*}
\cV_{\bm{\alpha}(t)}
&:=\mathrm{span}\{ \del_{\alpha_i} \,|\, i=1,\ldots,n \}\ \subset T_{\bm{\alpha}(t)}(T^*M), \\
\cJ_{\alpha}^t &:=(d\Phi_t)^{-1}(\cV_{\bm{\alpha}(t)})\ \subset T_{\alpha}(T^*M).
\end{align*}
The curve $(\cJ_{\alpha}^t)_{t \in (-\ve,\ve)}$
in the $n$-dimensional Grassmannian manifold in $T(T^*M)$
is called a \emph{Jacobi curve}, and represents how $T^*M$ is distorted
along the Hamiltonian flow $\Phi_t$.

Fix $\xi=\sum_{i=1}^n \xi^i \del_{\alpha_i} \in \cV_{\alpha}$.
Choose an arbitrary smooth curve $\bm{\xi}(t)=\sum_{i=1}^n \xi^i(t) \del_{\alpha_i} \in \cV_{\bm{\alpha}(t)}$
with $\bm{\xi}(0)=\xi$, and put $e_t :=(d\Phi_t)^{-1}(\bm{\xi}(t)) \in \cJ^t_{\alpha}$.
Since $e_t$ lives in the same linear space $T_{\alpha}(T^*M)$ for all $t$,
we can define $\dot{e}_t \in T_{\alpha}(T^*M)$ by differentiating the coefficients of $e_t$
(with respect to any linear coordinate of $T_{\alpha}(T^*M)$).
Recall that $\omega$ denotes the canonical symplectic form on $T^*M$.

\begin{lemma}\label{lm:quad}
We have
\[ \omega(\dot{e}_0,e_0)=\sum_{i,j=1}^n \xi^i \xi^j \bH_{\alpha_i \alpha_j}(\alpha). \]
In particular, $\omega(\dot{e}_0,e_0)$ is independent of the choice of the curve $\bm{\xi}$
and $\langle \xi,\xi \rangle_{\alpha}:=\omega(\dot{e}_0,e_0)$ defines an inner product of $\cV_{\alpha}$.
\end{lemma}

\begin{proof}
Since $e_0=\xi \in \cV_{\alpha}$ is in the vertical part of $T_{\alpha}(T^*M)$,
it is sufficient to calculate only the horizontal part of $\dot{e}_0$.
By differentiating $d\Phi_t(e_t)=\bm{\xi}(t)$ at $t=0$ and
noticing $\dot{\Phi}_t=\ora{\bH} \circ \Phi_t$ and $e_0=\xi$, we have
\[ \dot{e}_0
 +\sum_{i,j=1}^n \xi^j \{\bH_{\alpha_i \alpha_j}(\alpha) \del_{x^i} -\bH_{x^i \alpha_j}(\alpha) \del_{\alpha_i} \}
 =\sum_{i=1}^n \dot{\xi}^i(0) \del_{\alpha_i}. \]
Hence the horizontal part of $\dot{e}_0$ is $-\sum_{i,j=1}^n \xi^j \bH_{\alpha_i \alpha_j}(\alpha) \del_{x^i}$
and we obtain the claim.
$\qedd$
\end{proof}

\begin{remark}\label{rm:quad}
The above inner product should be compared with
\eqref{eq:g_ij} and \eqref{eq:g_V} in the Finsler case.
To be precise, $g_{\btau^*(\alpha)}$ of $T_xM$ coincides with
$\langle \cdot,\cdot \rangle_{\alpha}$ of $\cV_{\alpha}$ in Lemma~\ref{lm:quad}
via the Legendre transform and the canonical identification between $T_x^*M$ and $\cV_{\alpha}$.
\end{remark}

\begin{definition}[Canonical frames]\label{df:cano}
If a family of smooth curves $e^t_i \in \cJ^t_{\alpha}$ ($i=1,\ldots,n$) satisfies
\begin{equation}\label{eq:cano}
\omega(\dot{e}^t_i,e^t_j)=\delta_{ij}, \qquad
 \ddot{e}^t_i \in \cJ^t_{\alpha}\qquad \text{for all}\ t \in (-\ve,\ve),
\end{equation}
then we call the family $(e^t_i;\dot{e}^t_j)_{i,j=1}^n$
(referred simply by $(e^t_i)_{i=1}^n$ henceforth)
a \emph{canonical frame} along $\bm{\alpha}$.
\end{definition}

Indeed, the first condition in \eqref{eq:cano} ensures that
$(e^t_i)_{i=1}^n$ spans $\cJ^t_{\alpha}$
and $(e^t_i;\dot{e}^t_j)_{i,j=1}^n$ spans $T_{\alpha}(T^*M)$.
We also remark that
\[ \omega(\dot{e}^t_i,e^t_j) =\langle d\Phi_t(e_i^t),d\Phi_t(e_j^t) \rangle_{\bm{\alpha}(t)} \]
by \eqref{eq:omega} and Lemma~\ref{lm:quad}.
In Appendix~\ref{ssc:cano}, we will review how to construct a canonical frame
from an orthonormal basis with respect to $\langle \cdot,\cdot \rangle_{\bm{\alpha}(t)}$
along the recipe in \cite{Le2}.

Fix a canonical frame $E^t=(e_i^t)_{i=1}^n$ along $\bm{\alpha}$.

\begin{lemma}\label{lm:Lagsp}
We have $\omega(e^t_i,e^t_j)=0$ for all $i,j$ and $t$.
In particular, $\omega(e^t_i,\cJ^t_{\alpha})=0$ for all $i$ and $t$.
\end{lemma}

\begin{proof}
This follows from $\omega(e^0_i,e^0_j)=0$ (since $e^0_i \in \cV_{\alpha}$) and
\[ \del_t [\omega(e^t_i,e^t_j)] =\omega(\dot{e}^t_i,e^t_j)+\omega(e^t_i,\dot{e}^t_j)
 =\delta_{ij}-\delta_{ji}=0. \]
$\qedd$
\end{proof}

Hence $\cJ^t_{\alpha}$ is a Lagrangian subspace of $T_{\alpha}(T^*M)$
with respect to $\omega$.
The next lemma yields the uniqueness of a canonical frame
up to an orthogonal transformation.

\begin{lemma}{\rm (see \cite[Proposition~3.1]{Le2})}\label{lm:O^t}
Let $\ol{E}^t=(\bar{e}^t_i)_{i=1}^n$ be another canonical frame along $\bm{\alpha}$.
Then there exists an orthogonal matrix $O=(O_{ij})_{i,j=1}^n$ such that
$\ol{E}^t=OE^t$, i.e., $\bar{e}^t_i =\sum_{j=1}^n O_{ij}e^t_j$ for all $i$ and $t$.
In particular, $\ol{E}^t=E^t$ for all $t$ if $\ol{E}^0=E^0$.
\end{lemma}

\begin{proof}
For each $t$, the matrix $O^t$ given by $\ol{E}^t=O^tE^t$ is orthogonal since
\[ I_n =\omega(\dot{\ol{E}^t},\ol{E}^t) =O^t \omega(\dot{E}^t,E^t) (O^t)^{\sT}
 =O^t (O^t)^{\sT}, \]
where we used Lemma~\ref{lm:Lagsp} and $O^{\sT}$ is the transpose of $O$.
It also follows from Lemma~\ref{lm:Lagsp} that
\[ 0=\omega(\ddot{\ol{E}^t},\ol{E}^t)
 =2\dot{O}^t \omega(\dot{E}^t,E^t) (O^t)^{-1}
 =2\dot{O}^t (O^t)^{-1}. \]
Therefore $\dot{O}^t \equiv 0$ and $O=O^0$ is our desired matrix.
$\qedd$
\end{proof}

\begin{definition}[Curvature operator]\label{df:Hcurv}
For each $t \in (-\ve,\ve)$, the \emph{curvature operator}
$\bR_{\alpha}^t:\cJ^t_{\alpha} \lra \cJ^t_{\alpha}$ is defined by
\[ \bR^t_{\alpha}(e^t_i) :=-\ddot{e}^t_i. \]
\end{definition}

Note that $\bR_{\alpha}^t$ is a linear operator and is independent of
the choice of the canonical frame $(e_i^t)_{i=1}^n$ thanks to Lemma~\ref{lm:O^t}.
In Appendix~\ref{ssc:coord}, we explicitly calculate the curvature operator
in coordinates along \cite{Le2}.
The following property of $\bR_{\alpha}^t$ shows that the definition of $\bR^t_{\alpha}$
can be reduced to the case of $t=0$.

\begin{lemma}{\rm (see \cite[(25)]{Le2})}\label{lm:t=0}
For any $t$, we have $\bR^0_{\bm{\alpha}(t)}=d\Phi_t \circ \bR^t_{\alpha} \circ (d\Phi_t)^{-1}$.
\end{lemma}

\begin{proof}
Fix $t$ and observe that
\[ \tilde{e}_i^s :=d\Phi_t(e_i^{t+s}) \in d\Phi_t(\cJ_{\alpha}^{t+s})
 =(d\Phi_s)^{-1} (\cV_{\bm{\alpha}(t+s)}) =\cJ^s_{\bm{\alpha}(t)} \]
gives a canonical frame along $\bm{\alpha}(t+s)$.
Thus we obtain
\[ \bR^0_{\bm{\alpha}(t)}(\tilde{e}_i^0) =-\ddot{\tilde{e}}^0_i =-d\Phi_t(\ddot{e}_i^t)
 =d\Phi_t \big( \bR_{\alpha}^t (e^t_i) \big)
 =d\Phi_t \circ \bR^t_{\alpha} \circ (d\Phi_t)^{-1} (\tilde{e}^0_i). \]
$\qedd$
\end{proof}

\begin{lemma}\label{lm:symm}
The curvature operator $\bR^t_{\alpha}$ is symmetric in the sense that
\[ \left\langle d\Phi_t(e^t_i),d\Phi_t\big( \bR^t_{\alpha}(e^t_j) \big) \right\rangle_{\bm{\alpha}(t)}
 =\left\langle d\Phi_t\big( \bR^t_{\alpha}(e^t_i) \big),d\Phi_t(e^t_j) \right\rangle_{\bm{\alpha}(t)} \]
for all $i,j$ and $t$.
\end{lemma}

\begin{proof}
By Lemma~\ref{lm:t=0}, it suffices to see the claim at $t=0$.
We deduce from Lemma~\ref{lm:Lagsp} that $\omega(e^t_i,\ddot{e}^t_j) \equiv 0$,
and then $\omega(\dot{e}^t_i,\dot{e}^t_j) \equiv 0$ since
$\omega(e^t_i,\dot{e}^t_j) \equiv -\delta_{ij}$.
Thus we have
\[ \langle e^0_i,\ddot{e}^0_j \rangle_{\alpha}
 =\frac{1}{2}\{ \omega(\dot{e}^0_i,\ddot{e}^0_j) +\omega(\dddot{e}^0_j,e^0_i) \}
 =\frac{1}{2}\{ \omega(\dot{e}^0_i,\ddot{e}^0_j) -\omega(\ddot{e}^0_j,\dot{e}^0_i) \}
 =\omega(\dot{e}^0_i,\ddot{e}^0_j), \]
and
$\omega(\dot{e}^0_i,\ddot{e}^0_j)=-\omega(\ddot{e}^0_i,\dot{e}^0_j)
 =\omega(\dot{e}^0_j,\ddot{e}^0_i)$ shows the claim.
$\qedd$
\end{proof}

\begin{definition}[Ricci curvature]\label{df:HRic}
Define the \emph{Ricci curvature} $\Ric^{\bH}(\alpha) \in \R$ as the trace of
$\bR^0_{\alpha}:\cV_{\alpha} \lra \cV_{\alpha}$ with respect to the inner product
$\langle \cdot,\cdot \rangle_{\alpha}$ given in Lemma~\ref{lm:quad}.
We also set $\Ric^{\bH}:=0$ on $0_{T^*M}$.
\end{definition}

We remark that setting $\Ric^{\bH}(0)=0$ is reasonable since then $\bm{\alpha}$ is constant.
The weighted version can be introduced similarly to the Finsler case as follows
(recall Definition~\ref{df:FRic}).

\begin{definition}[Weighted Ricci curvature]\label{df:wHRic}
We fix a positive $\cC^{\infty}$-measure $\fm$ on $M$.
Along the action-minimizing curve $\eta(t):=\pi_M(\bm{\alpha}(t)) \in M$
($\pi_M:T^*M \lra M$ is the canonical projection),
decompose $\fm$ as $\fm=e^{-\psi}\vol_{\dot{\eta}}$,
where $\vol_{\dot{\eta}}$ is the volume form of $\bL_{v^i v^j}(\dot{\eta})$ given by
\[ \vol_{\dot{\eta}}= \sqrt{\det[\bL_{v^i v^j}(\dot{\eta})]}\, dx^1 \cdots dx^n\]
on $\eta$.
Then the \emph{weighted Ricci curvature} is defined by
\[ \Ric^{\bH}_N(\alpha)
 :=\Ric^{\bH}(\alpha)+(\psi \circ \eta)''(0)-\frac{(\psi \circ \eta)'(0)^2}{N-n} \]
for $N \in (n,\infty)$, and
\begin{align*}
\Ric^{\bH}_{\infty}(v) &:=\Ric^{\bH}(v) +(\psi \circ \eta)''(0), \\
\Ric^{\bH}_n(v) &:=\left\{
 \begin{array}{ll} \Ric^{\bH}(v)+(\psi \circ \eta)''(0) & \text{if}\ (\psi \circ \eta)'(0)=0, \\
 -\infty & \text{if}\ (\psi \circ \eta)'(0) \neq 0. \end{array} \right.
\end{align*}
We also set $\Ric^{\bH}_N :=0$ on $0_{T^*M}$ for all $N \in [n,\infty]$.
\end{definition}

\section{Laplacian}\label{sc:Lap}%%%%%%%%%%
%%%%%%%%%%%%%%%%%%%%%%

In this section, we introduce the natural nonlinear Laplacian $\Lap^{\bH}_{\fm}$
associated with the Hamiltonian $\bH$ and the reference measure $\fm$ on $M$
in a similar way to the Finsler case (see \cite{OShf}).
It will turn out that $\Lap^{\bH}_{\fm}$ coincides with the Laplacian $\Lap^{\bH}$
studied in \cite{Le2} (see also \cite{ALbl} for the sub-Riemannian case)
up to a difference depending on $\fm$.

\subsection{Gradient vectors, Laplacian and Hessian}\label{ssc:Lap}%%%%%%%%%

For a differentiable function $u$ on $M$, we call
\[ \Nabla u(x):=\btau^* (du_x) =\sum_{i=1}^n \bH_{\alpha_i}(du_x) \del_{x^i} \in T_xM \]
the \emph{gradient vector} of $u$ at $x \in M$.
For an open set $U \subset M$, we will use two kinds of \emph{Sobolev spaces}:
\begin{align*}
H^1(U;\bL) &:= \{ u \in L^2(U;\fm) \,|\, \text{weakly differentiable},\ \bL(\Nabla u) \in L^1(U;\fm) \}, \\
H^1(U;\bH) &:= \{ u \in L^2(U;\fm) \,|\, \text{weakly differentiable},\ \bH(du) \in L^1(U;\fm) \}.
\end{align*}
Clearly they coincide in the Riemannian or Finsler case.
We also introduce $H^1_{\loc}(U;\bL)$ and $H^1_{\loc}(U;\bH)$ similarly.
Note that $\cC^{\infty}_c(U) \subset H^1(U;\bL) \cap H^1(U;\bH)$
and $\cC^{\infty}(U) \subset H^1_{\loc}(U;\bL) \cap H^1_{\loc}(U;\bH)$.
We remark that $H^1(U;\bL)$ is not necessarily a linear space because of
the non-homogeneity of $\bL$
(same for $H^1(U;\bH)$, $H^1_{\loc}(U;\bL)$ and $H^1_{\loc}(U;\bH)$).
We know only that $H^1(U;\bH)$ (resp.\ $H^1_{\loc}(U;\bH)$)
is a convex subset of $L^2(U;\fm)$ (resp.\ $L^2_{\loc}(U;\fm)$) by the convexity of $\bH$.

Associated with $\fm$, define the \emph{divergence}
of a (weakly) differentiable vector field $V$ on $M$ with $\bL(V) \in L^1_{\loc}(M;\fm)$
in the weak form by
\[ \int_M \phi \div_{\fm} V \,d\fm =-\int_M d\phi(V) \,d\fm
 \quad \text{for all}\ \phi \in \cC^{\infty}_c(M). \]
Note that the right-hand side is well-defined since
\[ |d\phi(V)| \le \max\{\bH(d\phi),\bH(-d\phi)\} +\bL(V). \]
If $V$ is differentiable, then we can write down in coordinates as
\[ \div_{\fm}V =\sum_{i=1}^n \left\{ \frac{\del V^i}{\del x^i}
 -V^i \frac{\del \varsigma}{\del x^i} \right\}, \]
where $V=\sum_{i=1}^n V^i \del_{x^i}$ and $\fm=e^{-\varsigma}dx^1 \cdots dx^n$.

Now we define the (distributional, weighted) \emph{Laplacian} $\Lap^{\bH}_{\fm}$
acting on functions $u \in H^1_{\loc}(M;\bL)$
by $\Lap^{\bH}_{\fm} u:=\div_{\fm}(\Nabla u)$, namely
\[ \int_M \phi \Lap^{\bH}_{\fm} u \,d\fm =-\int_M d\phi(\Nabla u) \,d\fm
 \quad \text{for all}\ \phi \in \cC^{\infty}_c(M). \]
Note that our Laplacian is a negative operator in the sense that
\[ \int_M u \Lap^{\bH}_{\fm} u \,d\fm
 =-\int_M du(\Nabla u) \,d\fm =-\int_M \{ \bH(du)+\bL(\Nabla u) \} \,d\fm \le 0 \]
for all $u \in \cC^{\infty}_c(M)$ and equality holds if and only if $u$ is constant.
We remark that, even if $u \in \cC^{\infty}(M)$,
$\Nabla u$ may not be differentiable at points where $\Nabla u$ vanishes (Remark~\ref{rm:Legt}(a)).
On the set where $\Nabla u \neq 0$, we can calculate
\begin{align}
\Lap^{\bH}_{\fm} u &= \div_{\fm}\left( \sum_{i=1}^n \bH_{\alpha_i}(du) \del_{x^i} \right)
 \nonumber\\
&= \sum_{i=1}^n \left\{ \bH_{\alpha_i x^i}(du)
 +\sum_{j=1}^n \bH_{\alpha_i \alpha_j}(du) \frac{\del^2 u}{\del x^i \del x^j}
 -\bH_{\alpha_i}(du) \frac{\del \varsigma}{\del x^i} \right\}. \label{eq:Lapu}
\end{align}

The Laplacian studied in \cite{Le2} can be regarded as an unweighted version of $\Lap^{\bH}_{\fm}$.
Let us recall the definition in \cite[\S 4]{Le2}.
Given $u \in \cC^{\infty}(M)$, the image of the derivative
$d(du)_x:T_xM \lra T_{du_x}(T^*M)$ of $du:M \lra T^*M$ is
\[ P_x =\mathrm{span}\left\{ \del_{x^i} +\sum_{j=1}^n \frac{\del^2 u}{\del x^i \del x^j}(x) \del_{\alpha_j}
 \,\bigg|\, i=1,\ldots,n \right\} \]
(which is a Lagrangian subspace with respect to $\omega$).
We fix $x \in M$ with $du_x \neq 0$ and shall identify $P_x$ with the graph of a linear map
via a canonical frame $(e^t_i)_{i=1}^n$ along $\bm{\alpha}(t):=\Phi_t(du_x)$.
To be precise, we decompose as $T_{du_x}(T^*M)=\cV_{du_x} \times \cH_{du_x}$,
where
\[ \cH_{du_x}:=\mathrm{span}\{ \dot{e}_i^0 \,|\, i=1,\ldots,n \} \]
and is independent of the choice of a canonical frame by Lemma~\ref{lm:O^t}.
Choose a coordinate around $x$ such that $\bH_{\alpha_i \alpha_j}(du_x)=\delta_{ij}$ for all $i,j$,
and take the canonical frame $(e^t_i)_{i=1}^n$ with $e_i^0=\del_{\alpha_i}$.
(We prefer this coordinate than a simpler one with
$\bH_{\alpha_i \alpha_j} \circ \bm{\alpha} \equiv \delta_{ij}$
for the sake of visibility of the structure of the calculation.)
Then we have, by Proposition~\ref{pr:O^t}, \eqref{eq:e'} and \eqref{eq:Omega},
\begin{align*}
\dot{e}^0_i &= \sum_{j=1}^n (O^t_{ij} \bar{e}^t_j)'|_{t=0}
 =\frac{1}{2} \sum_{j=1}^n \Omega^0_{ij} \del_{\alpha_j} +\dot{\bar{e}}^0_i \\
&= -\del_{x^i}
 +\frac{1}{2}\sum_{j=1}^n \{ (\bH_{x^i \alpha_j} +\bH_{x^j \alpha_i})(du_x)
 +\dot{\xi}^i_j(0) +\dot{\xi}^j_i(0) \} \del_{\alpha_j}.
\end{align*}
Since $\sum_{k,l=1}^n \bH_{\alpha_k \alpha_l}(\bm{\alpha}(t)) \xi_i^k(t) \xi_j^l(t)=\delta_{ij}$
for all $t$ and $\bH_{\alpha_k \alpha_l}(du_x)=\delta_{kl}$,
we have
\[ \dot{\xi}^i_j(0) +\dot{\xi}^j_i(0)+[\bH_{\alpha_i \alpha_j} \circ \bm{\alpha}]'(0)=0. \]
Hence
\[ \dot{e}^0_i =-\del_{x^i}
 +\frac{1}{2}\sum_{j=1}^n \{ (\bH_{x^i \alpha_j} +\bH_{x^j \alpha_i})(du_x)
 -[\bH_{\alpha_i \alpha_j} \circ \bm{\alpha}]'(0) \} \del_{\alpha_j}, \]
so that $P_x$ is the graph of the linear map sending $-\dot{e}^0_i \in \cH_{du_x}$ to
\begin{equation}\label{eq:HHess}
\frac{1}{2} \sum_{j=1}^n \left\{ (\bH_{x^i \alpha_j}
 +\bH_{x^j \alpha_i})(du_x) -[\bH_{\alpha_i \alpha_j} \circ \bm{\alpha}]'(0)
 +2\frac{\del^2 u}{\del x^i \del x^j}(x) \right\} \del_{\alpha_j}.
\end{equation}
The negative of this map is called the \emph{Hessian}
$\Hess^{\bH}\! u(x):\cH_{du_x} \lra \cV_{du_x}$ of $u$ at $x$, and its trace
\begin{equation}\label{eq:LapH}
\Lap^{\bH} u(x) :=\sum_{i=1} \left\{ \bH_{x^i \alpha_i}(du_x)
 -\frac{1}{2}[\bH_{\alpha_i \alpha_i} \circ \bm{\alpha}]'(0)
 +\frac{\del^2 u}{(\del x^i)^2}(x) \right\}
\end{equation}
with respect to the canonical frame is called the Laplacian in \cite{Le2}.

To compare $\Lap^{\bH}_{\fm}u$ with $\Lap^{\bH}u$, similarly to Definition~\ref{df:wHRic},
let us decompose $\fm$ along the action-minimizing curve
$\eta(t):=\pi_M(\bm{\alpha}(t))$ as
\[ \fm=e^{-\psi} \sqrt{\det[\bL_{v^i v^j}(\dot{\eta})]} \,dx^1 \cdots dx^n. \]
Then we observe from \eqref{eq:Lapu} that,
since $\sum_{i=1}^n \bH_{\alpha_i}(du_x)\del_{x^i} =\dot{\eta}(0)$,
\[ \Lap^{\bH}_{\fm} u(x) =\sum_{i=1}^n
 \left\{ \bH_{\alpha_i x^i}(du_x) +\frac{\del^2 u}{(\del x^i)^2}(x) \right\}
 -\left[ \psi \circ \eta -\frac{1}{2} \log(\det [\bL_{v^i v^j} \circ \dot{\eta}]) \right]'(0). \]
Note that
\begin{align*}
\Big[\! \log(\det[\bL_{v^i v^j} \circ \dot{\eta}]) \Big]'(0)
&=\Big[\! \det[\bL_{v^i v^j} \circ \dot{\eta}] \Big]'(0)
 =-\Big[\! \det[\bH_{\alpha_i \alpha_j} \circ \bm{\alpha}] \Big]'(0) \\
&= -\sum_{i=1}^n [\bH_{\alpha_i \alpha_i} \circ \bm{\alpha}]'(0).
\end{align*}
Therefore we have
\begin{equation}\label{eq:LapHm}
\Lap^{\bH}_{\fm} u(x) =\Lap^{\bH} u(x) -(\psi \circ \eta)'(0),
\end{equation}
so that $\Lap^{\bH}_{\fm}u$ and $\Lap^{\bH}u$ coincide
up to the term depending on the weight function $\psi$.
In other words, $\Lap^{\bH}$ can be interpreted as the weighted Laplacian $\Lap^{\bH}_{\fm}$
with respect to a measure $\fm$ such that $\psi$ is constant along any action-minimizing curve $\eta$,
whereas such a measure does not necessarily exist even on a Finsler manifold (Remark~\ref{rm:wRic}).

\begin{remark}\label{rm:Hess}
It can be checked by hand that the above Hessian coincides with $\Nabla^2 u \in T^*M \otimes TM$
used in \cite{OSbw} to study the Bochner--Weitzenb\"ock formula in the Finsler setting.
To see this, recall from \cite[Lemma~2.3]{OSbw} that
\[ \Nabla^2 u (\del_{x^i}|_x) =\sum_{j=1}^n
 \left\{ \frac{\del^2 u}{\del x^i \del x^j}
 -\sum_{k=1}^n \Gamma^k_{ij}(\Nabla u) \frac{\del u}{\del x^k} \right\}(x) \del_{x^j}|_x, \]
where $g_{ij}(\Nabla u(x))=\delta_{ij}$ is still assumed.
We calculate by using the notations in \cite{OSbw} as, at $x$,
\begin{align*}
&\sum_{k=1}^n \Gamma^k_{ij} (\Nabla u) \frac{\del u}{\del x^k} \\
&= \frac{1}{2} \sum_{k=1}^n \left\{ \left(
 \frac{\del g_{jk}}{\del x^i} +\frac{\del g_{ik}}{\del x^j} -\frac{\del g_{ij}}{\del x^k} \right)
 +\sum_{l=1}^n \frac{\del g_{ij}}{\del v^l} N^l_k \right\} (\Nabla u)
 \frac{\del u}{\del x^k} \\
&= -\frac{1}{2} \sum_{k=1}^n \left(
 \frac{\del g^{jk}}{\del x^i} +\frac{\del g^{ik}}{\del x^j} -\frac{\del g^{ij}}{\del x^k} \right)
 (du) \frac{\del u}{\del x^k}
 +\frac{1}{2}\sum_{l=1}^n \frac{\del g_{ij}}{\del v^l} (\Nabla u) G^l (\Nabla u),
\end{align*}
where $(g^{ij}(du))$ is the inverse matrix of $(g_{ij}(\Nabla u))$.
It follows from the homogeneity that
\[ \bH_{x^i \alpha_j}(du_x)
 =\sum_{k=1}^n \bH_{x^i \alpha_j \alpha_k}(du_x) \frac{\del u}{\del x^k}(x)
 =\sum_{k=1}^n \frac{\del g^{jk}}{\del x^i}(du_x) \frac{\del u}{\del x^k}(x). \]
Moreover, since $\dot{\eta}(0)=\Nabla u(x)$, we have
\[ [\bH_{\alpha_i \alpha_j} \circ \bm{\alpha}]'(0)
 =\sum_{k=1}^n \left\{ \frac{\del g^{ij}}{\del x^k}(du_x) \frac{\del u}{\del x^k}(x)
 -\frac{\del g_{ij}}{\del v^k}\big( \dot{\eta}(0) \big) \ddot{\eta}^k(0) \right\}. \]
Consequently, the geodesic equation
$\ddot{\eta}^k+G^k(\dot{\eta}) \equiv 0$ shows that
\[ \Nabla^2 u (\del_{x^i}|_x) =\frac{1}{2} \sum_{j=1}^n \left\{
 2\frac{\del^2 u}{\del x^i \del x^j}(x) +(\bH_{x^i \alpha_j} +\bH_{x^j \alpha_i})(du_x)
 -[\bH_{\alpha_i \alpha_j} \circ \bm{\alpha}]'(0) \right\} \del_{x^j}|_x. \]
Compare this with \eqref{eq:HHess}.
\end{remark}

\subsection{Energy and harmonic functions}\label{ssc:harm}%%%%%%%%

Let $U \subset M$ be an open set.
Define the \emph{energy functional} $\cE_U:H^1_{\loc}(U;\bH) \lra [0,\infty]$ by
\[ \cE_U(u):=\int_U \bH(du) \,d\fm. \]

\begin{lemma}\label{lm:lsc}
The energy functional $\cE_U$ is lower semi-continuous on $L^2(U;\fm)$.
Namely, for any sequence $\{u_i\}_{i \in \N} \subset L^2(U;\fm) \cap H^1_{\loc}(U;\bH)$
converging to $u \in L^2(U;\fm) \cap H^1_{\loc}(U;\bH)$ in $L^2(U;\fm)$, it holds
\[ \cE_U(u) \le \liminf_{i \to \infty} \cE_U(u_i). \]
\end{lemma}

\begin{proof}
Given $\ve>0$, let $W \subset U$ be a compact set such that
$\cE_W(u) \ge \cE_U(u)-\ve$ (including the case where both energies are infinite).
We cover $W$ with finitely many, mutually disjoint open sets $\{U_k\}$
(up to an $\fm$-negligible set) such that each $U_k$ is diffeomorphic to the unit ball in $\R^n$.
Then Serrin's classical theorem (\cite{Se}, see also \cite{FGM}) is applicable
on each $U_k$ to obtain
\[ \cE_W(u) =\sum_k \cE_{U_k}(u)
 \le \sum_k \liminf_{i \to \infty} \cE_{U_k}(u_i) \le \liminf_{i \to \infty} \cE_U(u_i). \]
We complete the proof by letting $\ve \to 0$.
$\qedd$
\end{proof}

We say that $u \in H^1_{\loc}(U;\bL)$ is \emph{weakly harmonic} on $U$
if $\Lap^{\bH}_{\fm} u \equiv 0$ on $U$ in the weak sense, that is,
\[ \int_U d\phi(\Nabla u) \,d\fm =0 \qquad \text{for all}\ \phi \in \cC_c^{\infty}(U). \]

\begin{lemma}\label{lm:wharm}
Suppose that $u+\phi \in H^1(U';\bH)$ holds for any open subset $U' \subset U$,
$u \in H^1(U';\bH)$ and $\phi \in \cC^{\infty}_c(U')$.
Then a function $u \in H^1(U;\bH) \cap H^1_{\loc}(U;\bL)$ is weakly harmonic on $U$ if and only if
\[ \cE_U(u) =\inf\{ \cE_U(u+\phi) \,|\, \phi \in \cC_c^{\infty}(U) \}. \]
Similarly, $u \in H^1_{\loc}(U;\bH) \cap H^1_{\loc}(U;\bL)$
is weakly harmonic on $U$ if and only if,
for any relatively compact open set $U' \subset U$,
\[ \cE_{U'}(u) =\inf\{ \cE_{U'}(u+\phi) \,|\, \phi \in \cC_c^{\infty}(U') \}. \]
\end{lemma}

\begin{proof}
The convexity of $\bH$ yields that, for any $t \in (0,1)$,
\begin{align*}
\bH(du) -\bH\big( d(u-\phi) \big) &\le \frac{\bH\big( d(u+t\phi) \big) -\bH(du)}{t}
 \le \bH\big( d(u+\phi) \big) -\bH(du), \\
\bH(du) -\bH\big( d(u-\phi) \big) &\le \frac{\bH(du) -\bH\big( d(u-t\phi) \big)}{t}
 \le \bH\big( d(u+\phi) \big) -\bH(du).
\end{align*}
Hence the dominated convergence theorem implies that
\[ \del_t [\cE_{U'}(u+t\phi)]|_{t=0} =\int_{U'} d\phi(\Nabla u) \,d\fm, \]
which completes the proof of the both assertions.
$\qedd$
\end{proof}

Note that $u \in H^1(U;\bH)$ was necessary to ensure $\cE_U(u)<\infty$,
while $u \in H^1_{\loc}(U;\bL)$ was used to make $\Lap_{\fm}^{\bH} u$ well-defined.
The hypothesis of the lemma is fulfilled by, for example,
the $p$-homogeneous deformations of Finsler (or Riemannian) Hamiltonians
(see Subsection~\ref{ssc:F^p}).
In that case our Laplacian coincides with the (weighted) $p$-Laplacian.

\subsection{Riccati equation}\label{ssc:Ricc}%%%%%%

In \cite[\S 4]{Le2}, Lee showed a Riccati type equation with respect to the Laplacian
$\Lap^{\bH}$ in \eqref{eq:LapH}
(see also \cite{ALge}, \cite{ALbl} and \cite{LLZ} for the sub-Riemannian case).
We repeat his argument for completeness, and derive a generalization of
the Bochner--Weitzenb\"ock formula along the line of \cite{OSbw} in the next subsection.

Let $u \in \cC^{\infty}(M)$ and fix $x \in M$ with $du_x \neq 0$.
On a small neighborhood $U$ of $x$ on which $du \neq 0$,
let us consider the solution $(u_t)_{t \in (-\ve,\ve)} \subset \cC^{\infty}(U)$
to the \emph{Hamilton--Jacobi equation}
\begin{equation}\label{eq:HJ}
\del_t u_t +\bH(du_t)=0, \qquad u_0=u|_U.
\end{equation}
A geometric (or dynamical) meaning of \eqref{eq:HJ} is that, for each $y \in U$,
\begin{equation}\label{eq:HJam}
\Phi_t(du_y) =(du_t)_{\cT_t(y)}
\end{equation}
holds (as far as $\cT_t(y) \in U$), where $t \longmapsto \cT_t(y)$
is the action-minimizing curve with $\del_t[\cT_t(y)]|_{t=0}=\Nabla u(y)$.

Put $\eta(t):=\cT_t(x)$
and let $(e^t_i)_{i=1}^n$ be a canonical frame along $\bm{\alpha}(t):=(du_t)_{\eta(t)}$.
Then, for each fixed $\tau$,
\[ \tilde{e}_i^{t,\tau}:=d\Phi_{\tau}(e_i^t) \in d\Phi_{\tau}(\cJ_{du_x}^t)
 =\cJ_{\bm{\alpha}(\tau)}^{t-\tau} \]
gives a canonical frame along $\bm{\alpha}(\tau+t)$
(similarly to the proof of Lemma~\ref{lm:t=0}).
We write down the map $d(du)_x:T_xM \lra T_{du_x}(T^*M)$ as, for each $t$,
\begin{equation}\label{eq:ddu}
d(du)_x \circ d\pi_M(\dot{e}^0_i)
 =\sum_{j=1}^n \{ A_{ij}(t) e^t_j +B_{ij}(t) \dot{e}^t_j \},
 \qquad i=1,\ldots,n,
\end{equation}
where $A(t)=(A_{ij}(t))$ and $B(t)=(B_{ij}(t))$ are $(n \times n)$-matrices.
Applying $d\Phi_t$ yields
\begin{equation}\label{eq:dPhi}
d\Phi_t \circ d(du)_x \circ d\pi_M(\dot{e}^0_i)
 =\sum_{j=1}^n \{ A_{ij}(t) \tilde{e}^{t,t}_j +B_{ij}(t) \dot{\tilde{e}}^{t,t}_j \}.
\end{equation}
On the one hand, since $\tilde{e}^{t,t}_j \in \cV_{\bm{\alpha}(t)}$,
\eqref{eq:dPhi} implies
\begin{align*}
d\cT_t \circ d\pi_M(\dot{e}^0_i)
&= d(\pi_M \circ \Phi_t \circ du)_x \circ d\pi_M(\dot{e}^0_i)
 =d\pi_M \left( \sum_{j=1}^n \{ A_{ij}(t) \tilde{e}^{t,t}_j +B_{ij}(t) \dot{\tilde{e}}^{t,t}_j \} \right) \\
&= \sum_{j=1}^n B_{ij}(t) d\pi_M(\dot{\tilde{e}}^{t,t}_j).
\end{align*}
On the other hand, differentiating \eqref{eq:HJam} at $y=x$ gives
$d\Phi_t \circ d(du)_x =d(du_t)_{\eta(t)} \circ (d\cT_t)_x$.
By combining these with \eqref{eq:dPhi}, we have
\[ \sum_{j=1}^n B_{ij}(t) d(du_t)_{\eta(t)} \circ d\pi_M(\dot{\tilde{e}}^{t,t}_j)
 =\sum_{j=1}^n \{ A_{ij}(t) \tilde{e}^{t,t}_j +B_{ij}(t) \dot{\tilde{e}}^{t,t}_j \}. \]
This shows that $\Hess^{\bH}\! u_t(\eta(t))=-B(t)^{-1}A(t)$ in the frame
$(\tilde{e}^{t,t}_i;\dot{\tilde{e}}^{t,t}_j)$.

By differentiating \eqref{eq:ddu} in $t$, we obtain
\[ 0=\sum_{j=1}^n \left[ \left\{ \dot{A}_{ij}(t) -\sum_{k=1}^n B_{ik}(t) \bR_{kj}(t) \right\}e^t_j
 +\{ A_{ij}(t) +\dot{B}_{ij}(t) \} \dot{e}^t_j \right], \]
where we set $\bR_{\alpha}^t(e^t_j)=\sum_{k=1}^n \bR_{jk}(t)e^t_k$.
This implies
\[ \dot{B}(t)=-A(t), \qquad \dot{A}(t)=B(t) \bR(t). \]
We consequently obtain the \emph {matrix Riccati equation} (\cite[(31)]{Le2})
\begin{equation}\label{eq:mRicc}
\del_t \big[\! \Hess^{\bH}\! u_t\big( \eta(t) \big) \big]
 +\left[ \Hess^{\bH}\! u_t\big( \eta(t) \big) \right]^2 +\bR(t) =0.
\end{equation}
Taking the trace yields,
by the symmetry of $\Hess^{\bH}\!u_t$ (see \eqref{eq:HHess}) and Lemma~\ref{lm:t=0},
\begin{equation}\label{eq:Ricc}
\del_t \big[ \Lap^{\bH} u_t\big( \eta(t) \big) \big]
 +\left\| \Hess^{\bH}\! u_t\big( \eta(t) \big) \right\|_{\HS(du_t)}^2
 +\Ric^{\bH}\! \left( (du_t)_{\eta(t)} \right) =0,
\end{equation}
where $\|\cdot\|_{\HS(du_t)}$ denotes the \emph{Hilbert--Schmidt norm}
with respect to a canonical frame along $\bm{\alpha}(t)=(du_t)_{\eta(t)}$.

\subsection{Bochner--Weitzenb\"ock formula}\label{ssc:BW}%%%%%%%

Taking the reference measure $\fm$ into account, we readily observe from
\eqref{eq:Ricc} and \eqref{eq:LapHm} that
\begin{equation}\label{eq:wRicc}
\del_t \big[ \Lap^{\bH}_{\fm} u_t\big( \eta(t) \big) \big]
 +\left\| \Hess^{\bH}\! u_t\big( \eta(t) \big) \right\|_{\HS(du_t)}^2
 +\Ric^{\bH}_{\infty}\! \left( (du_t)_{\eta(t)} \right) =0.
\end{equation}
From this, similarly to the Finsler case,
one can derive the \emph{Bochner--Weitzenb\"ock formula}.

\begin{theorem}[Bochner--Weitzenb\"ock formula]\label{th:BW}
Let $u \in \cC^{\infty}(M)$.
At any $x \in M$ with $du_x \neq 0$, we have
\begin{align}
\Delta^{du}_{\fm} \big( \bH(du) \big) -d(\Lap^{\bH}_{\fm} u)(\Nabla u)
&=\Ric^{\bH}_{\infty}(du) +\left\| \Hess^{\bH}\! u \right\|_{\HS(du)}^2, \label{eq:BW}\\
\Delta^{du}_{\fm} \big( \bH(du) \big) -d(\Lap^{\bH}_{\fm} u)(\Nabla u)
&\ge \Ric^{\bH}_N(du) +\frac{(\Lap^{\bH}_{\fm} u)^2}{N} \label{eq:N-BW}
\end{align}
for $N \in [n,\infty)$, where
\[ \Delta^{du}_{\fm} f :=\div_{\fm}\left(
 \sum_{i,j=1}^n \bH_{\alpha_i \alpha_j}(du) \frac{\del f}{\del x^j} \del_{x^i} \right). \]
\end{theorem}

\begin{proof}
Calculate the first term in \eqref{eq:wRicc} at $t=0$ as
\[ \del_t \big[ \Lap^{\bH}_{\fm} u_t\big( \eta(t) \big) \big]|_{t=0}
 =d(\Lap^{\bH}_{\fm} u)_x \big( \dot{\eta}(0) \big)
 +\div_{\fm} \left( \sum_{i=1}^n \del_t \big[ \bH_{\alpha_i} \big( (du_t)_x \big) \big]|_{t=0} \del _{x^i} \right). \]
Then \eqref{eq:BW} follows from $\dot{\eta}(0)=\Nabla u(x)$ and \eqref{eq:HJ} since
\[ \del_t \big[ \bH_{\alpha_i}\big( (du_t)_x \big) \big]|_{t=0}
 =-\sum_{i,j=1}^n \bH_{\alpha_i \alpha_j}\big( (du)_x \big)
 \frac{\del[\bH(du)]}{\del x^j}(x). \]
One can derive \eqref{eq:N-BW} from \eqref{eq:BW} in a standard way
with the help of \eqref{eq:LapHm}
(see the proof of \cite[Theorem~3.3]{OSbw} for instance).
$\qedd$
\end{proof}

It is easily seen that $\Delta^{du}_{\fm} f$ is well-defined on
$\{x \in M \,|\, du_x \neq 0\}$.

\begin{remark}\label{rm:BW}
(a)
The formulas in Theorem~\ref{th:BW} coincide with those in
\cite[Theorem~3.3]{OSbw} on Finsler manifolds.
We remark that, different from the Finsler setting (see \cite[Lemma~2.4]{OShf}),
$\Delta^{du}_{\fm} u=\Lap^{\bH}_{\fm}u$ does not necessarily hold
due to the non-homogeneity of $\bH$.

(b)
In the Finsler case, one can treat functions with lower regularity
by passing to the weak (integrated) formulations of \eqref{eq:BW} and \eqref{eq:N-BW}
(\cite[Theorem~3.6]{OSbw}).
It is a more delicate issue for general Hamiltonians.
Furthermore, we could obtain Bakry--\'Emery and Li--Yau gradient estimates in \cite[\S 4]{OSbw}
as applications of  the (weak form of) Bochner--Weitzenb\"ock formulas
(see also \cite{Ospl}, \cite{WX}, \cite{Xi} for further applications).
In these proofs, however, we were essentially indebted to the homogeneity of the Finsler metric
(Euler's theorem \cite[Theorem~2.2]{OSbw} to be precise)
and it is unclear whether these gradient estimates can be generalized to general Hamiltonians.
\end{remark}

\section{Examples}\label{sc:expl}%%%%%%%%%%
%%%%%%%%%%%%%%%%%%%%%%

This section is devoted to discussing several examples of Hamiltonians
and calculating the curvature and Laplacian of them.

\subsection{Finsler Hamiltonians}\label{ssc:FinsH}%%%%%%%%%%%%%

Let $(M,F)$ be a Finsler manifold as in Subsection~\ref{ssc:Fins}.

\begin{example}\label{ex:FinsH}
Let $\bL^F(v):=F(v)^2/2$ be the canonical Lagrangian induced from $F$,
and $\bH^F$ be the associated Hamiltonian.
Then the Hamiltonian curvature operator $\bR^0_{\alpha}:\cV_{\alpha} \lra \cV_{\alpha}$
for $\alpha \in T_x^*M \setminus \{0\}$ coincides with the Finsler curvature operator
$\bR^F(\cdot,\btau^*(\alpha))\btau^*(\alpha):T_xM \lra T_xM$
(with the reference vector $\btau^*(\alpha)$)
by identifying $\cV_{\alpha}$ and $T_xM$ (as in Remark~\ref{rm:quad}).
In particular, we have $\Ric^{\bH}(\alpha)=\Ric^F(\btau^*(\alpha))$.
\end{example}

One can verify this coincidence by the Riccati equation \eqref{eq:mRicc}
(whose Finsler version can be found in \cite[Lemmas~3.1, 3.2]{OSbw}) for instance
(see also \cite[\S 11]{Le2}).
Recall that there is a useful interpretation of the Finsler curvature
by using Riemannian structures induced from vector fields satisfying a certain condition
(see Subsection~\ref{ssc:Fins}).
For general Lagrangians, however, it seems difficult to obtain an analogous interpretation
because even $2\bL(v)=\sum_{i,j=1}^n \bL_{v^i v^j}(v)v^i v^j$ fails.

\subsection{Natural mechanical Hamiltonians}\label{ssc:mech}%%%%%%%%%%

Let $(M,g)$ be a Riemannian manifold and $Z \in \cC^{\infty}(M)$.

\begin{example}[Agrachev]\label{ex:mech}
The \emph{natural mechanical Hamiltonian}
\[ \bH(\alpha):=\bH^g(\alpha) +Z(x), \quad \alpha \in T^*_xM, \]
yields
\begin{equation}\label{eq:mech}
\bR_{\alpha}^0=\bR^g\big( \cdot,\btau^*(\alpha) \big) \btau^*(\alpha) +\Hess^g Z(x)
\end{equation}
by identifying $\cV_{\alpha}$ on the LHS and $T_xM$ on the RHS as in Example~\ref{ex:FinsH},
where $\btau^*$ is the Legendre transform common to $\bH^g$ and $\bH$.
In particular, we have $\Ric^{\bH}(\alpha)=\Ric^g(\btau^*(\alpha)) +\Lap^g Z(x)$.
\end{example}

We remark that $\bH$ is allowed to be negative since it is $\cC^{\infty}$ on whole $T^*M$.
Since $Z$ is constant on each $T_xM$, we immediately find that
$\bL(v)=\bL^g(v)-Z(x)$ for $v \in T_xM$,
and that the Legendre transform of $\bH$ coincides with that of $\bH^g$.
One can also see that the Hessian and Laplacian are common to $\bH^g$ and $\bH$.
Then, in the Riccati equation \eqref{eq:mRicc}, the term $\Hess^g Z$ in \eqref{eq:mech}
comes from the difference between the Hamilton--Jacobi equations \eqref{eq:HJ} of $\bH^g$ and $\bH$.
To see this, given $u \in \cC^{\infty}(M)$ and $x \in M$ with $du_x \neq 0$,
let $(u_t)$ and $(\tilde{u}_t)$ be the solutions to \eqref{eq:HJ}
around $x$, with $u_0=\tilde{u}_0=u$, with respect to $\bH^g$ and $\bH$, respectively.
By comparing the matrix representations of $\Hess^{\bH} \tilde{u}_t(x)$
and $\Hess^g u_t(x)$ as in \eqref{eq:HHess}, we have
\begin{align*}
&\del_t [\Hess^{\bH} \tilde{u}_t(x)]|_{t=0} -\del_t [\Hess^g u_t(x)]|_{t=0} \\
&= -\frac{1}{2}\sum_{k=1}^n \{
 \bH^g_{x^i \alpha_j \alpha_k}+\bH^g_{x^j \alpha_i \alpha_k}-\bH^g_{\alpha_i \alpha_j x^k} \}(x)
 \frac{\del Z}{\del x^k}(x) -\frac{\del^2 Z}{\del x^i \del x^j}(x) \\
&= -\frac{\del^2 Z}{\del x^i \del x^j}(x)
 +\sum_{k=1}^n \Gamma^k_{ij}(x)\frac{\del Z}{\del x^k}(x) \\
&= -\Hess^g Z(x).
\end{align*}
We used $\bH_{\alpha_i \alpha_j \alpha_k} \equiv 0$ in the first equality.
This yields \eqref{eq:mech} (with $\alpha=du_x$).
One can alternatively show \eqref{eq:mech} by the direct computation in \eqref{eq:Rij}.

The weighted curvature can be treated similarly as follows.
Let $\fm$ be a positive $\cC^{\infty}$-measure on $M$.
In the current situation, we can write down as $\fm=e^{-\psi}\vol_g$ globally.
Since the Legendre transform is common to $\bH^g$ and $\bH$,
the weighted Laplacian $\Lap^{\bH}_{\fm}$ coincides with $\Lap^g_{\fm}$.
Thus we find from \eqref{eq:Lapu} that, for $(u_t)$ and $(\tilde{u}_t)$ as above,
\begin{align*}
&\del_t [\Lap^{\bH}_{\fm} \tilde{u}_t(x)]|_{t=0} -\del_t [\Lap^g_{\fm} u_t(x)]|_{t=0} \\
&=-\sum_{i,j=1}^n \left\{ \bH^g_{\alpha_i x^i \alpha_j} \frac{\del Z}{\del x^j}
 +\bH^g_{\alpha_i \alpha_j} \frac{\del^2 Z}{\del x^i \del x^j}
 -\bH^g_{\alpha_i \alpha_j} \frac{\del \varsigma}{\del x^i} \frac{\del Z}{\del x^j} \right\} (x) \\
&= -\Lap^g_{\fm} Z(x).
\end{align*}
We used Euler's theorem in the last line.
Hence we have, by comparing \eqref{eq:wRicc} for $\bH^g$ and $\bH$,
\[ \Ric^{\bH}_{\infty}(\alpha)=\Ric^g_{\infty}\!\big( \btau^*(\alpha) \big) +\Lap^g_{\fm} Z(x). \]
Then, since $\bL_{v^i v^j}=\bL^g_{v^i v^j}$ and the difference
$\Ric^{\bH}_N(\alpha) -\Ric^{\bH}_{\infty}(\alpha)$ involves
only the first order derivatives of the weight function $\psi$,
we immediately obtain
\[ \Ric^{\bH}_N(\alpha)=\Ric^g_N\!\big( \btau^*(\alpha) \big) +\Lap^g_{\fm} Z(x) \]
for $N \in [n,\infty)$ as well.

\subsection{Convex deformations of Finsler Hamiltonians}\label{ssc:Fconv}%%%%%%%%

Let $(M,F)$ be a Finsler manifold and $h:\R \lra [0,\infty)$ be a
$\cC^{\infty}$-function with $h(0)=h'(0)=0$, $h'' >0$ on $(0,\infty)$,
as well as $\lim_{t \to \infty}h'(t)=\infty$.
In particular, $h>0$ on $(0,\infty)$
and $h':[0,\infty) \lra [0,\infty)$ is bijective.
Now, let us consider the Hamiltonian $\bH(\alpha):=h(F^*(\alpha))$.
Since
\[ \btau^*(\alpha)
 =\sum_{i=1}^n h'\big( F^*(\alpha) \big) \frac{\del F^*}{\del \alpha_i}(\alpha) \del_{x^i}
 =\frac{h'(F^*(\alpha))}{F^*(\alpha)} \btau^*_F(\alpha) \]
on $T^*M \setminus 0_{T^*M}$, where $\btau^*_F$ is the Legendre transform
with respect to $F$, we have
\begin{equation}\label{eq:hLap}
\Lap^{\bH}_{\fm} u =\div_{\fm} \left( \frac{h'(F^*(du))}{F^*(du)} \Nabla^F u \right)
 =\frac{h'(F^*(du))}{F^*(du)} \Lap^F_{\fm} u
 +d \left[ \frac{h'(F^*(du))}{F^*(du)} \right] (\Nabla^F u)
\end{equation}
on $\{ x \in M \,|\, du_x \neq 0 \}$.
One can see this also by calculating \eqref{eq:Lapu}.
The calculation of the curvature is a little more involved,
for that we go back to the definition of the curvature operator.
The following proposition shows that the ratio of $\Ric^{\bH}(\alpha)$
to $\Ric^F(\btau^*_F(\alpha))$ depends only on $F^*(\alpha)$ and $h$.

\begin{proposition}\label{pr:Fconv}
Let $\bH(\alpha)=h(F^*(\alpha))$ be as above.
Then we have
\[ \bR^0_{\alpha}=c(\alpha)^2 \wt{\bR}^0_{\alpha}, \qquad
 \Ric^{\bH}(\alpha)=c(\alpha)^2 \Ric^F\! \big( \btau^*_F (\alpha) \big) \]
for any $\alpha \in T^*M \setminus 0_{T^*M}$,
where $\wt{\bR}^0_{\alpha}$ denotes the curvature operator
$($in the sense of Definition~$\ref{df:Hcurv})$ with respect to $\bH^F$
and we set $c(\alpha):=h'(F^*(\alpha))/F^*(\alpha)$.
It also holds that $\Ric^{\bH}_N(\alpha)=c(\alpha)^2 \Ric^F_N(\btau^*_F (\alpha))$
for any $N \in [n,\infty]$.
\end{proposition}

\begin{proof}
Put $c:=c(\alpha)$ and $\tilde{\bm{\alpha}}(t):=\Phi_t^F(\alpha)$ for $t \in (-\ve,\ve)$,
and take a coordinate $(x^i)_{i=1}^n$ such that
$dx^1=F^*(\tilde{\bm{\alpha}}(t))^{-1} \tilde{\bm{\alpha}}(t)$ and
$(\del_{x^i})_{i=1}^n$ is orthonormal with respect to
$g_{\btau^*_F(\bm{\alpha}(t))}$ for all $t$.
Then we can take a canonical frame $(\tilde{e}_i^t)_{i=1}^n$
along $\tilde{\bm{\alpha}}(t)$ with respect to $\bH^F$
such that $\tilde{e}_1^t=(d\Phi_t^F)^{-1}(\del_{\alpha_1}|_{\tilde{\bm{\alpha}}(t)})$.
(Start from $\bar{e}_i^t:=(d\Phi_t^F)^{-1}(\del_{\alpha_i}|_{\tilde{\bm{\alpha}}(t)})$
in Appendix~\ref{ssc:cano} and observe that $\Omega^t_{i1} \equiv 0$ for all $i$
from $\ddot{\bar{e}}^t_1 \equiv 0$ and Lemma~\ref{lm:orth},
then we have $O^t_{11} \equiv 1$ and
$O^t_{1i}=O^t_{i1} \equiv 0$ for $i \ge 2$ by decomposing the equation in Proposition~\ref{pr:O^t}.)

Switching to $\bH$, we first observe from
\[ \ora{\bH}\big( \tilde{\bm{\alpha}}(t) \big)
 =\frac{h'(F^*(\tilde{\bm{\alpha}}(t)))}{F^*(\tilde{\bm{\alpha}}(t))}
 \ora{\bH}^F\big(\tilde{\bm{\alpha}}(t) \big)
 =c \ora{\bH}^F\big(\tilde{\bm{\alpha}}(t) \big) \]
that $\Phi_t(\alpha)=\Phi_{ct}^F(\alpha)$.
Put $\bm{\alpha}(t):=\Phi_t(\alpha)=\tilde{\bm{\alpha}}(ct)$
and note that, by $\bH_{\alpha_i}=[h'(F^*)/F^*] \bH_{\alpha_i}^F$,
$\bH_{\alpha_i \alpha_j}^F(\bm{\alpha}(t))=\delta_{ij}$,
the choice of the coordinate and Euler's theorem,
\begin{align*}
\bH_{\alpha_i \alpha_j}\big( \bm{\alpha}(t) \big)
&=c\delta_{ij} \quad \text{for}\ i \ge 2\ \text{or}\ j \ge 2,\\
\bH_{\alpha_1 \alpha_1}\big( \bm{\alpha}(t) \big)
&= c+\left[ \frac{h''(F^*) F^* -h'(F^*)}{(F^*)^2} \frac{\del F^*}{\del \alpha_1} \bH^F_{\alpha_1} \right]
 \big( \bm{\alpha}(t) \big) \\
&= c+\left[ \frac{h''(F^*) F^* -h'(F^*)}{(F^*)^2} \frac{F^*}{F^*} \frac{2\bH^F}{F^*} \right]
 \big( \bm{\alpha}(t) \big) \\
&= h''\big( F^*(\alpha) \big).
\end{align*}
Hence
\[ e_1^t :=\frac{1}{\sqrt{h''(F^*(\alpha))}} \tilde{e}_1^{ct}, \qquad
e_i^t :=\frac{1}{\sqrt{c}} \tilde{e}_i^{ct} \quad \text{for}\ i \ge 2 \]
give a canonical frame along $\bm{\alpha}(t)$ with respect to $\bH$.
Now we have
\[ \bR_{\alpha}^t(e^t_1) =-\ddot{e}^t_1
 =\frac{c^2}{\sqrt{h''(F^*(\alpha))}} \wt{\bR}_{\alpha}^{ct}(\tilde{e}_1^{ct})
 =c^2 \wt{\bR}_{\alpha}^{ct}(e_1^t), \]
and similarly $\bR_{\alpha}^t(e^t_i)=c^2 \wt{\bR}_{\alpha}^{ct}(e_i^t)$ for $i \ge 2$.
These show $\bR^0_{\alpha}=c^2 \wt{\bR}^0_{\alpha}$
as well as $\Ric^{\bH}(\alpha)=c^2 \Ric^F(\btau^*_F(\alpha))$.

The weighted version immediately follows from the above calculations.
To be precise, denoting by $\tilde{\psi}$ and $\psi$ the weight functions
as in Definition~\ref{df:wHRic} with respect to $\bH^F$ and $\bH$,
we find $\psi \circ \pi_M \circ \bm{\alpha}(t)=\tilde{\psi} \circ \pi_M \circ \tilde{\bm{\alpha}}(ct)+C$
with some constant $C$ (depending on $F^*(\alpha)$ and $h$).
This yields $\Ric^{\bH}_N(\alpha)=c^2 \Ric^F_N(\btau^*_F(\alpha))$.
$\qedd$
\end{proof}

For instance, if $h(t)=(at)^2/2$ with some $a>0$, then we have
\[ \Ric^{\bH}(\alpha) =a^4 \Ric^F\! \big( \btau^*_F(\alpha) \big)
 =a^2 \Ric^F\! \big( \btau^*_F(\alpha) \big) \cdot \frac{\bH(\alpha)}{\bH^F(\alpha)}. \]
When we deform the Lagrangian as $\bL(v):=h(F(v))$ instead of the Hamiltonian,
we observe from
\[ \btau(v)=\frac{h'(F(v))}{F(v)} \btau_F(v), \qquad
 h\big( F(v) \big) +\bH\big( \btau(v) \big) =[\btau(v)](v) =h'\big( F(v) \big) F(v) \]
that $F^*(\btau(v))=h'(F(v))$ and
\[ \bH(\alpha) =F^*(\alpha) \cdot (h')^{-1}\big( F^*(\alpha) \big)
 -h \circ (h')^{-1}\big( F^*(\alpha) \big). \]
Since the function $\bar{h}(t):=t(h')^{-1}(t)-h \circ (h')^{-1}(t)$ is convex
(note that $\bar{h}'=(h')^{-1}$)
and satisfies $\bar{h}(0)=\bar{h}'(0)=0$, we can apply Proposition~\ref{pr:Fconv} and obtain
\[ \Ric^{\bH}_N(\alpha) =\left( \frac{(h')^{-1}(F^*(\alpha))}{F^*(\alpha)} \right)^2
 \Ric^F_N\! \big( \btau^*_F(\alpha) \big). \]

\subsection{Homogeneous deformations of Finsler Hamiltonians}\label{ssc:F^p}%%%%%%

Let $(M,F)$ be a Finsler manifold again.
One of the most important examples of convex deformations of $F^*$
is the \emph{$p$-homogeneous deformation}:
\[ \bH_p(\alpha) :=\frac{F^*(\alpha)^p}{p},\qquad 1<p<\infty. \]
The corresponding Lagrangian is $\bL_q(v) :=F(v)^q/q$,
where $q=p/(p-1)$ is the dual exponent of $p$.
By \eqref{eq:hLap}, $\bH_p$ leads to the (Finsler analogue of) famous \emph{$p$-Laplacian}
\[ \Lap_{\fm}^{\bH_p} u=\div_{\fm}\big( F^*(du)^{p-2} \cdot \Nabla^F u). \]
One can perform better analysis for $\Lap_{\fm}^{\bH_p}$ thanks to the homogeneity,
although we do not pursue that direction in this paper
(except for Example~\ref{ex:p-gf}, we refer to recent works \cite{Ke1}, \cite{Ke2} instead).
Note also that $\Ric^{\bH_p}_N(\alpha)=F^*(\alpha)^{2(p-2)}\Ric^F_N(\btau^*_F(\alpha))$
by Proposition~\ref{pr:Fconv}.
This suggests that the behavior of $\Lap_{\fm}^{\bH_p} u$ depends on
how large $F^*(du)$ is.

\section{Laplacian comparison theorem}\label{sc:Lcomp}%%%%%%
%%%%%%%%%%%%%%%%%%%%%%

We return to a general Hamiltonian $\bH$ on $M$ and show a generalization of
the Laplacian comparison theorem associated with the weighted Ricci curvature $\Ric^{\bH}_N$
(suggested in \cite[Remark~4.4]{Le2},
see also \cite{ALbl} for a related work in the sub-Riemannian setting).
We also refer to \cite[Theorem~5.2]{OShf} for the Finsler case and to \cite{Gi}
for the more general case of metric measure spaces enjoying the curvature-dimension condition.
Our Laplacian comparison theorem would be compared with the Bonnet--Myers type theorem
of Agrachev and Gamkrelidze (\cite[Theorem in \S 4]{AG}, \cite[Theorem~2.1]{Agr}).
Throughout the section, we assume the forward completeness of the Hamiltonian flow $\Phi_t$.

\subsection{Laplacian comparison}\label{ssc:Lcomp}%%%%%%%%%%%%%%%

Fix $z \in M$ and $c>0$.
We introduce the (dual of) \emph{exponential map of scale $c$},
$\exp_z^c:T_z^* M \setminus \{0\} \lra M$, by $\exp_z^c(\alpha):=\eta(a^{-1})$
where $a>0$ is chosen as $\bH(a\alpha)=c$ and $\eta$ is the action-minimizing curve
with $\dot{\eta}(0)=\btau^*(a\alpha)$.
In other words, $\exp_z^c(\alpha)=\pi_M(\Phi_{a^{-1}}(a\alpha))$.
Set also $\exp_z^c(0):=z$.

\begin{remark}\label{rm:exp}
In the Riemannian and Finsler cases,
the map $\exp_z^c$ coincides with the ordinary exponential map
$\exp_z$ for all $c>0$ (via the Legendre transform as usual).
For the $p$-homogeneous deformation $\bH_p$ of a Finsler metric as in Subsection~\ref{ssc:F^p},
it holds that
\[ \exp_z^{c'}(\alpha)=\exp_z^c \bigg( \left( \frac{c'}{c} \right)^{(p-2)/p} \alpha \bigg)
 \quad \text{for}\ c,c'>0,\ \alpha \in T_z^*M \]
since $\bH_p(a\alpha)=c'$ implies
\[ \bH_p \bigg( \left( \frac{c}{c'} \right)^{(p-1)/p} a \cdot \left( \frac{c'}{c} \right)^{(p-2)/p} \alpha \bigg)
 =\bH_p \bigg( \left( \frac{c}{c'} \right)^{1/p} a\alpha \bigg)
 =\frac{c}{c'} \bH_p(a\alpha) =c \]
and
\[ \btau^*\bigg( \left( \frac{c}{c'} \right)^{1/p} a\alpha \bigg)
 =\left( \frac{c}{c'} \right)^{(p-1)/p} \btau^*(a\alpha). \]
\end{remark}

Put
\[ D_z^c:=\{ t\alpha \in T_z^*M \,|\, \bH(\alpha)=c,\, 0<t<T_{\alpha} \},
 \qquad \cD_z^c:=\exp_z^c(D_z^c), \]
where $T_{\alpha} \in (0,\infty]$ is the supremum of $T>0$ for which
$d\!\exp_z^c$ is invertible on $\{t\alpha\}_{t \in (0,T)}$ and, for all $t \in (0,T)$,
there is no $\alpha' \in T_z^*M \setminus \{\alpha\}$ satisfying $\bH(\alpha')=c$
as well as $\exp_z^c(t\alpha')=\exp_z^c(t\alpha)$.
Then we define the one-form $\bm{\alpha}$ on $\cD_z^c$ by
$\bm{\alpha}(\exp_z^c(t\alpha))=\btau(\dot{\eta}(t))=\Phi_t(\alpha)$ for $t\alpha \in D_z^c$,
where $\bH(\alpha)=c$ and $\eta$ is the action-minimizing curve with
$\dot{\eta}(0)=\btau^*(\alpha)$.
Integrating along $\eta(t)=\exp_z^c(t\alpha)$ gives the function $u_z^c \in \cC^{\infty}(\cD_z^c)$ satisfying
$\lim_{x \to z} u_z^c(x)=0$ and $du_z^c=\bm{\alpha}$.
Since $(u_z^c \circ \eta)'(t)=\bm{\alpha}(\dot{\eta}(t))=c+\bL(\dot{\eta}(t))$, we have
\[ u_z^c \big( \eta(t) \big) =ct +\int_0^t \bL(\dot{\eta}) \,ds. \]

\begin{theorem}[Laplacian comparison]\label{th:Lcomp}
Assume that $\Ric^{\bH}_N(\alpha) \ge K$ holds for some $K \in \R$, $N \in [n,\infty)$,
and all $\alpha \in T^*M$ with $\bH(\alpha)=c$.
Then we have, for any $z \in M$, $\alpha \in T_z^*M$ with $\bH(\alpha)=c$, and
the action-minimizing curve $\eta:[0,T_{\alpha}] \lra \ol{\cD_z^c}$ with $\dot{\eta}(0)=\btau^*(\alpha)$,
\begin{equation}\label{eq:Lcomp}
\Lap^{\bH}_{\fm} u_z^c\big( \eta(t) \big) \le \sqrt{NK} \cot\left( \sqrt{\frac{K}{N}}t \right)
 \quad \text{for all}\ t \in (0,T_{\alpha})
\end{equation}
if $K>0$.
It similarly holds that
$\Lap^{\bH}_{\fm} u_z^c(\eta(t)) \le \sqrt{-NK} \coth(\sqrt{-K/N}t)$ if $K<0$,
and $\Lap^{\bH}_{\fm} u_z^c(\eta(t)) \le N/t$ if $K=0$.
\end{theorem}

\begin{proof}
It follows from \eqref{eq:wRicc} with $u_t=u_z^c -ct$ and Theorem~\ref{th:BW} that
\[ \del_t \big[ \Lap^{\bH}_{\fm} u_z^c \big( \eta(t) \big) \big]
 \le -\Ric^{\bH}_N \big( (du_z^c)_{\eta(t)} \big) -\frac{[\Lap^{\bH}_{\fm} u_z^c(\eta(t))]^2}{N} \]
for $t \in (0,T_{\alpha})$.
We have $\Ric^{\bH}_N((du_z^c)_{\eta(t)})=\Ric^{\bH}_N(\bm{\alpha}(\eta(t))) \ge K$
by assumption, therefore
\[ \del_t \big[ \Lap^{\bH}_{\fm} u_z^c \big( \eta(t) \big) \big]
 \le -K -\frac{[\Lap^{\bH}_{\fm} u_z^c(\eta(t))]^2}{N} \]
for all $t \in (0,T_{\alpha})$.

Put $u(t):=\Lap^{\bH}_{\fm} u_z^c(\eta(t))$ for brevity and compare it with
$\tilde{u}(t):=N\bs'_{K,N}(t)/\bs_{K,N}(t)$, where
\[ \bs_{K,N}(t) := \left\{
 \begin{array}{cl}
 \sqrt{N/K} \sin (\sqrt{K/N}t) & {\rm if}\ K>0, \\
 t & {\rm if}\ K=0, \\
 \sqrt{-N/K} \sinh (\sqrt{-K/N}t) & {\rm if}\ K<0.
 \end{array} \right. \]
Observe that $\tilde{u}'=-K-\tilde{u}^2/N$ and
$\lim_{t \downarrow 0} \bs_{K,N}(t)^2 \tilde{u}(t)=0$.
Then the desired inequality $u \le \tilde{u}$ follows from
$\lim_{t \downarrow 0} \bs_{K,N}(t)^2 u(t)=0$ and the calculation (see \cite[(3.10)]{WW})
\begin{align*}
[\bs_{K,N}^2 (u-\tilde{u})]'
&= 2\bs_{K,N} \bs'_{K,N} (u-\tilde{u}) +\bs_{K,N}^2 (u-\tilde{u})'
 \le 2\frac{\bs_{K,N}^2 \tilde{u}}{N}(u-\tilde{u}) -\bs_{K,N}^2 \frac{u^2-\tilde{u}^2}{N} \\
&= -\frac{\bs_{K,N}^2}{N}(u-\tilde{u})^2 \le 0.
\end{align*}

Finally, $\lim_{t \downarrow 0} \bs_{K,N}(t)^2 u(t)=0$ can be seen as follows.
Choose a local coordinate $(x^i)_{i=1}^n$ of a small neighborhood $U$ of $z$
such that every action-minimizing curve $\sigma$
with $\sigma(0)=z$ and $\bH(\btau(\dot{\sigma})) \equiv c$
is represented as $\sigma(t)=(tb^i)_{i=1}^n$ with $\sum_{i=1}^n (b^i)^2=1$.
Then it holds that $\Nabla u_z^c(x) =r^{-1} \sum_{i=1}^n x^i \del_{x^i}$ on $U \setminus \{z\}$,
where  $r=\sqrt{\sum_{i=1}^n (x^i)^2}$.
Hence we find $\lim_{t \downarrow 0} t^2 u(t)=0$ by
\[ r^2\sum_{i=1}^n \del_{x^i}[r^{-1}x^i] =nr -\sum_{i=1}^n r^{-1}(x^i)^2
 =(n-1)r \to 0 \]
as $r \downarrow 0$.
$\qedd$
\end{proof}

\begin{remark}\label{rm:Lcomp}
In the Riemannian and Finsler cases, $u_z^c/\sqrt{2c}$ coincides with
the distance function $d_z$ from $z$, and \eqref{eq:Lcomp} is improved to
\begin{equation}\label{eq:sLcom}
\Lap_{\fm} d_z \le \sqrt{(N-1) \frac{K}{2c}}
 \cot\left( \sqrt{\frac{1}{N-1} \frac{K}{2c}}d_z \right)
\end{equation}
by decomposing $\Lap_{\fm} d_z$ into the radial and spherical directions from $z$.
It is unclear if a similar improvement can be done for general Hamiltonians
(see also the Bonnet--Myers type theorem of Agrachev and Gamkrelidze).
The point is that $u_z^c(\eta(t))$ is not proportional to $t$
since $(u_z^c \circ \eta)'(t)=c+\bL(\dot{\eta}(t))$
and $\bL(\dot{\eta})$ is not necessarily constant.
We also remark that even the stronger inequality \eqref{eq:sLcom}
does not imply the curvature bound $\Ric^{\bH}_N \ge K$
(see \cite[Remark~5.6]{StII} for a simple example,
and \cite{Ju} for a related work on the gap between the measure contraction property
and the curvature-dimension condition in Heisenberg groups).
\end{remark}

\subsection{Measure contraction property}\label{ssc:MCP}%%%%%%%%%%%%

A geometric counterpart to the Laplacian comparison is the \emph{measure contraction property}
(see \cite{Omcp}, \cite[\S 5]{StII} for the precise definition on metric measure spaces,
and \cite{ALge}, \cite{LLZ} for related works on sub-Riemannian manifolds).
To state it, we fix $c>0$, $z \in M$, $\alpha \in T_z^*M$
and $\eta:[0,T_{\alpha}] \lra M$ as in Theorem~\ref{th:Lcomp},
and take a local coordinate $(x^i)_{i=1}^n$ of a neighborhood $U$ of $\eta(T)$
with $T \in (0,T_{\alpha})$ such that $\Nabla u_z^c \equiv \del_{x^1}$.
We also introduce $\varsigma:U \lra \R$ by $\fm=e^{\varsigma} dx^1 \cdots dx^n$.
Then the measure contraction property we consider is that the ratio
\[ \frac{e^{\varsigma(\eta(t))}}{\bs_{K,N}(t)^N} \]
is non-increasing in $t$.
This is clearly equivalent to
\[ (\varsigma \circ \eta)'(t) \le N\frac{\bs'_{K,N}(t)}{\bs_{K,N}(t)}. \]
Since $\Nabla u_z^c \equiv \del_{x^1}$ and $\dot{\eta}(t)=\Nabla u_z^c(\eta(t))$,
we observe
\[ \div_{\fm}(\Nabla u_z^c) \big( \eta(t) \big) =\frac{\del \varsigma}{\del x^1} \big( \eta(t) \big)
 =(\varsigma \circ \eta)'(t). \]
Therefore the above measure contraction property is equivalent to
the Laplacian comparison \eqref{eq:Lcomp}:
\[ \Lap^{\bH}_{\fm} u_z^c \big( \eta(t) \big) \le N\frac{\bs'_{K,N}(t)}{\bs_{K,N}(t)}. \]

\begin{remark}\label{rm:Lee}
In \cite[Theorems~2.2, 2.10]{Le2}, Lee showed convexity estimates
of the relative and R\'enyi entropies along smooth optimal transports.
These can be interpreted as generalizations of the curvature-dimension condition
(recall the last paragraph in Section~\ref{sc:prel}),
and has applications to various monotonicity formulas along flows in Riemannian metrics
related to the Ricci flow (\cite[\S 2]{Le2}).
Precisely, this approach recovers the curvature-dimension conditions
$\CD(K,\infty)$ and $\CD(0,N)$ for Riemannian or Finsler manifolds.
The condition $\CD(K,N)$ with $K \neq 0$ and $N<\infty$ is more delicate and seems difficult
to treat for general Hamiltonians for the same reasoning as Remark~\ref{rm:Lcomp}.
\end{remark}

\section{Optimal transports and functional inequalities}\label{sc:opt}%%%%%%%
%%%%%%%%%%%%%%%%%%%%%%

In this section, we briefly recall some properties of optimal transports
measured by Lagrangian cost functions.
Then, assuming that the relative entropy is convex along all optimal transports,
we obtain functional inequalities along the lines of \cite{OV}, \cite{LV2}.
See the comprehensive reference \cite{Vi2} for optimal transport theory.

Let $M$ be compact throughout the section for simplicity,
and $\bL$ be our Lagrangian.
For later use, we fix an auxiliary Riemannian metric $g$ of $M$.

\subsection{Optimal transport theory}\label{ssc:opt}%%%%%%%%%%

For $x,y \in M$ and $T>0$, we consider the \emph{Lagrangian cost function} given by
\[ c^{\bL}_T(x,y) :=\inf_{\eta} \int_0^T \bL(\dot{\eta}) \,dt, \]
where the infimum is taken over all $\cC^1$-curves $\eta:[0,T] \lra M$ with $\eta(0)=x$
and $\eta(T)=y$.
Note that $c^{\bL_g}_T(x,y)=d_g(x,y)^2/(2T)$, where $d_g$ is the distance function
with respect to $g$.

Denote by $\cP(M)$ the set of all Borel probability measures on $M$.
Given $\mu,\nu \in \cP(M)$, $\pi \in \cP(M \times M)$ is called a \emph{coupling}
of $\mu$ and $\nu$ if its projections to the first and second arguments
coincide with $\mu$ and $\nu$, i.e.,
$\pi[A \times M]=\mu[A]$ and $\pi[M \times A]=\nu[A]$ for any Borel set $A \subset M$.
Denote by $\Pi(\mu,\nu)$ the set of all couplings of $\mu$ and $\nu$,
which is nonempty since the product measure $\mu \times \nu$
is clearly a coupling of $\mu$ and $\nu$.
Then the \emph{transport cost} from $\mu$ to $\nu$ (measured by $\bL$) is defined as
\begin{equation}\label{eq:W_c}
C^{\bL}_T(\mu,\nu) :=\inf_{\pi \in \Pi(\mu,\nu)} \int_{M \times M} c^{\bL}_T(x,y) \,\pi(dxdy),
\end{equation}
which is finite since $M$ is compact.
A coupling $\pi \in \Pi(\mu,\nu)$ achieving the infimum
in \eqref{eq:W_c} is called an \emph{optimal coupling} of $\mu$ and $\nu$.
Under mild assumptions, there exists a unique optimal coupling 
whose support is drawn as the graph of some map $\cT:M \lra M$
(an \emph{optimal transport} from $\mu$ to $\nu$).
This fundamental fact was first shown for the canonical Lagrangian on Euclidean spaces by Brenier~\cite{Br},
and extended to compact Riemannian manifolds by McCann~\cite{Mc2},
and then to general Lagrangians on noncompact spaces by Bernard and Buffoni~\cite{BeBu},
Fathi and Figalli~\cite{FF} and Villani~\cite[Chapter~10]{Vi2}.

\begin{theorem}{\rm (see \cite[Theorem~10.28]{Vi2})}\label{th:Monge}
Take $\mu,\nu \in \cP(M)$ such that $\mu \ll \fm$.
Then there exists a semi-convex function $\varphi$
$($with respect to $g)$ such that the map
\[ \cT_t(x):=\pi_M\big( \Phi_t(d\varphi_x) \big) \]
gives a unique optimal coupling $\pi=(\id_M \times \cT_T)_{\sharp}\mu$ of $\mu$ and $\nu$, i.e.,
$(\cT_T)_{\sharp}\mu=\nu$ and
\[ C^{\bL}_T(\mu,\nu)=\int_M c^{\bL}_T\big( x,\cT_T(x) \big) \,\mu(dx). \]
\end{theorem}

Denoted by $f_{\sharp}\mu$ was the push-forward measure of $\mu$ by a map $f$.
Note that \cite[Theorem~10.28]{Vi2} is certainly available by \cite[Proposition~10.15, Theorem~10.26]{Vi2}.
Observe also that, by \cite[Proposition~10.15(ii), (10.20)]{Vi2},
\[ d\varphi_x=-d\big[ c_T^{\bL}\big( \cdot,\cT_T(x) \big) \big]_x
 =\btau\big( \dot{\eta}(0) \big) \]
for $\mu$-almost every $x$,
where $\eta:[0,T] \lra M$ is the unique action-minimizing curve from $x$ to $\cT_T(x)$.
When we put $\mu_t:=(\cT_t)_{\sharp}\mu$,
it clearly holds that $C^{\bL}_T(\mu,\nu)=C^{\bL}_t(\mu,\mu_t) +C^{\bL}_{T-t}(\mu_t,\nu)$ for all $t \in (0,T)$.
We have $\mu_t \ll \fm$ for every $t \in (0,T)$ (see \cite[Theorem~8.7]{Vi2}).
The semi-convexity of $\varphi$ implies that $\varphi$ is twice differentiable $\mu$-almost everywhere,
so that $\cT_t$ is differentiable $\mu$-almost everywhere.

\begin{remark}\label{rm:C^2}
Lagrangians are assumed to be $\cC^2$ on whole $TM$ at some places in \cite{Vi2}.
However, it causes no problem because our Lagrangian fails to be $\cC^2$ only on the zero section $0_{TM}$
which is isolated in the Euler--Lagrange flow (recall the paragraph following Definition~\ref{df:Lag}).
\end{remark}

To state the \emph{change of variables} (or the \emph{Jacobian equation}) for the map $\cT_t$,
we define the \emph{Jacobian determinant} of $\cT_t$ at $x \in M$ with respect to $\fm$ by
\[ \bD_{\fm}[\cT_t](x)
 :=\lim_{r \downarrow 0} \frac{\fm[\cT_t(B^g(x,r))]}{\fm[B^g(x,r)]} \]
if the limit exists, where $B^g(x,r)$ is the open ball of center $x$ and radius $r$
with respect to the Riemannian distance induced from our auxiliary Riemannian metric $g$.
Note that, at points where $\cT_t$ is differentiable, $\bD_{\fm}[\cT_t]$
does not depend on the choice of $g$.

\begin{theorem}[Change of variables]\label{th:cov}
Let $\mu,\nu$ and $\cT_t$ be as in Theorem~$\ref{th:Monge}$,
and put $\mu_t:=(\cT_t)_{\sharp}\mu=\rho_t \fm$.
Then we have, for any $t \in (0,T)$ and nonnegative measurable function
$f:[0,\infty) \lra [0,\infty)$ with $f(0)=0$,
\begin{equation}\label{eq:cov}
\int_M f(\rho_t) \,d\fm
 =\int_M f \left( \frac{\rho_0}{\bD_{\fm}[\cT_t]} \right) \bD_{\fm}[\cT_t] \,d\fm.
\end{equation}
If in addition $\nu \ll \fm$, then \eqref{eq:cov} holds also at $t=T$.
\end{theorem}

\begin{proof}
This is a simple application of \cite[Theorem~11.3]{Vi2}.
Decompose $\fm$ as $\fm=e^{-\psi} \vol_g$.
Applying \cite[Theorem~11.3]{Vi2} to $F(x,s)=e^{-\psi(x)}f(se^{\psi(x)})$ yields
\[ \int_M e^{-\psi} f(\rho_t) \,d\!\vol_g
 =\int_M e^{-\psi(\cT_t)} f \left(
 \frac{\rho_0 e^{-\psi}}{\bD^g_{\fm}[\cT_t]} e^{\psi(\cT_t)} \right) \bD^g_{\fm}[\cT_t] \,d\!\vol_g, \]
where $\bD^g_{\fm}[\cT_t]$ denotes the Jacobian determinant with respect to $\vol_g$.
Together with $\bD_{\fm}[\cT_t]=e^{\psi-\psi(\cT_t)}\bD^g_{\fm}[\cT_t]$, we have
\[ \int_M f(\rho_t) \,d\fm
 =\int_M f \left( \frac{\rho_0}{\bD_{\fm}[\cT_t]} \right) \bD_{\fm}[\cT_t] \,d\fm. \]
(Precisely, one needs to separate the discussion into
$M_{\varphi}:=\{ x \in M \,|\, d\varphi_x \neq 0 \}$ and $M \setminus M_{\varphi}$.
On $M \setminus M_{\varphi}$, $\cT_t$ is the identity map and
$\bD_{\fm}[\cT_t]=1$ for all $t$, $\mu$-almost everywhere.)
$\qedd$
\end{proof}

\subsection{Functional inequalities}\label{ssc:func}%%%%%%%%

For $\mu \in \cP(M)$, the \emph{relative entropy} with respect to $\fm$ is defined by
\[ \Ent_{\fm}(\mu) :=\int_M \rho \log\rho \,d\fm \]
if $\mu=\rho\fm \ll \fm$, and $\Ent_{\fm}(\mu):=\infty$ otherwise.
Note that $\Ent_{\fm}(\mu) \ge -\log(\fm[M])$ by Jensen's inequality.

Lee in \cite[Theorem~2.2]{Le2} showed that, along an optimal transport $(\mu_t)_{t \in [0,T]}$
as in Theorem~\ref{th:Monge} such that $\varphi \in \cC^{\infty}(M)$,
it holds that
\[ \del_t^2 [\Ent_{\fm}(\mu_t)]
 \ge \int_M \Ric_{\infty}^{\bH}\!\big( \Phi_t(d\varphi_x) \big) \,\mu_0(dx)
 \quad \text{for all}\ t. \]
Therefore, if $\Ric_{\infty}^{\bH}(\alpha) \ge K\bH(\alpha)$ for some $K \in \R$ and all $\alpha \in T^*M$,
then we have
\begin{equation}\label{eq:Kconv}
\Ent_{\fm}(\mu_{(1-t)a+bt}) \le (1-t)\Ent_{\fm}(\mu_a) +t\Ent_{\fm}(\mu_b)
 -\frac{K}{2}(1-t)t (b-a)^2 \int_M \bH(d\varphi_x) \,\mu_0(dx)
\end{equation}
for all $0 \le a<b \le T$ and $t \in [0,1]$.
We set
\[ C^{\bH}_T(\mu_0,\mu_1):=T\int_M \bH(d\varphi_x) \,\mu_0(dx)
 =\int_0^T \!\int_M \bH\big( \Phi_t(d\varphi_x) \big) \,\mu_0(dx) dt \]
for later use.
Note that $C^{\bH}_T=C^{\bL}_T$ in the Riemannian and Finsler cases.

We say that $\Ent_{\fm}$ is \emph{$K$-convex} for $K \in \R$ and $T>0$ if \eqref{eq:Kconv} holds
for all (not necessarily smooth) optimal transports $(\mu_t)_{t \in [0,T]}$ as in Theorem~\ref{th:Monge}.

\begin{example}\label{ex:CDK}
(a)
For Finsler manifolds, thanks to the homogeneity, the $K$-convexity for some $T>0$ implies
the $K$-convexity for all $T>0$.
In this case, the $2K$-convexity of $\Ent_{\fm}$ with respect to $\bH^F(\alpha)=F^*(\alpha)^2/2$
is the very definition of the curvature-dimension condition
$\CD(K,\infty)$ and is equivalent to $\Ric^F_{\infty}(v) \ge KF(v)^2$ (recall Subsection~\ref{ssc:Fins}).

(b)
Assume that a Finsler manifold $(M,F,\mu)$ satisfies $\Ric^F_{\infty}(v) \ge KF(v)^2$.
Then, for the $p$-homogeneous deformation $\bH_p(\alpha)=F^*(\alpha)^p/p$, we have
\[ \Ric^{\bH_p}_{\infty}(\alpha)
 =F^*(\alpha)^{2(p-2)} \Ric^F_{\infty}\big( \btau^*_F(\alpha) \big)
 \ge KF^*(\alpha)^{2(p-2)+2} \\
 =pK\bH_p(\alpha) F^*(\alpha)^{p-2}. \]
To estimate $F^*(\alpha)$ for $K \neq 0$,
observe that $TF^*(d\varphi) \le \diam_F M$ along all optimal transports $(\mu_t)_{t \in [0,T]}$,
where $\diam_F M :=\sup_{x,y \in M}d_F(x,y)$ is the diameter of $(M,F)$.
Therefore $\Ent_{\fm}$ is $(pK(\diam_F M/T)^{p-2})$-convex with respect to $\bH_p$ for $T>0$
when
(1) $K=0$ and $p \in (1,\infty)$; or
(2) $K>0$ and $p \in (1,2)$; or
(3) $K<0$ and $p \in (2,\infty)$.
\end{example}

The following proposition will be useful.

\begin{proposition}[Directional derivatives of $\Ent_{\fm}$]\label{pr:dEnt}
Suppose that $\Ent_{\fm}$ is $K$-convex for some $K \in \R$ and $T>0$,
and take $\mu=\rho\fm \in \cP(M)$ such that $\rho \in H^1(M;\bH)$.
Then we have, for any optimal transport $\mu_t=(\cT_t)_{\sharp}\mu$,
$t \in [0,T]$, as in Theorem~$\ref{th:Monge}$,
\begin{equation}\label{eq:dEnt}
\liminf_{t \downarrow 0} \frac{\Ent_{\fm}(\mu_t)-\Ent_{\fm}(\mu)}{t}
 \ge \int_M \frac{d\rho}{\rho}(\Nabla\varphi) \,d\mu.
\end{equation}
\end{proposition}

\begin{proof}
It follows from Theorem~\ref{th:cov} with $f(t)=t \log t$ that
\begin{align*}
\Ent_{\fm}(\mu_t) &= \int_M \rho_t \log \rho_t \,d\fm
 =\int_M \frac{\rho}{\bD_{\fm}[\cT_t]} \log \left( \frac{\rho}{\bD_{\fm}[\cT_t]} \right) \bD_{\fm}[\cT_t] \,d\fm \\
&=\Ent_{\fm}(\mu) -\int_M \log(\bD_{\fm}[\cT_t]) \,d\mu.
\end{align*}
(To be precise, we applied Theorem~\ref{th:cov} to $\max\{f,0\}$ and $\max\{-f,0\}$
and took their difference.)
Hence, by localizing the assumption \eqref{eq:Kconv},
we find that $-\log(\bD_{\fm}[\cT_t](x))$ is $K$-convex in $t$ for $\mu$-almost every $x$.
Thus the monotone convergence theorem shows that
\[ \lim_{t \downarrow 0} \int_M \frac{\log(\bD_{\fm}[\cT_t])}{t} \,d\mu
 =\int_M \lim_{t \downarrow 0} \frac{\log(\bD_{\fm}[\cT_t])}{t} \,d\mu
 =\int_M \lim_{t \downarrow 0} \frac{\bD_{\fm}[\cT_t] -1}{t} \,d\mu. \]
We see by calculation
\begin{equation}\label{eq:dT_t}
\lim_{t \downarrow 0} \frac{\bD_{\fm}[\cT_t] -1}{t} =\Lap^{\bH}_{\fm} \varphi
 \qquad \text{$\mu$-almost everywhere}.
\end{equation}
Indeed, observe in the calculation in Subsection~\ref{ssc:Ricc} that $d\cT_t=B(t)$ and $B(0)=I_n$.
Then we have
\[ (\det B)'(0)=\trace(B^{-1} \dot{B})(0) =-\trace(B^{-1} A)(0)
 =\Lap^{\bH} \varphi. \]
Plugging the effect of the measure $\fm$ into this yields \eqref{eq:dT_t}
(recall \eqref{eq:LapHm}).

Notice that the singular part of $\Lap^{\bH}_{\fm} \varphi$ is non-negative
by looking at the second order term of \eqref{eq:Lapu}.
Finally, we obtain \eqref{eq:dEnt} by integration by parts,
see Steps~1--3 in \cite[Theorem~23.14]{Vi2} for details.
$\qedd$
\end{proof}

For $\mu=\rho\fm \in \cP(M)$, define
\[ I_{\fm}(\mu):=\int_M \bH(-d[\log\rho]) \,d\mu\ \in [0,\infty]. \]
This quantity can be thought of as the \emph{Fisher information} adapted to our context.

\begin{theorem}[Functional inequalities]\label{th:funct}
Assume that $\fm[M]=1$ and that $\Ent_{\fm}$ is $K$-convex
in the sense of \eqref{eq:Kconv} for some $K,T>0$.
\begin{enumerate}[{\rm (i)}]
\item \emph{(Talagrand inequality)}
For any $\mu \in \cP(M)$, it holds
\[ C^{\bH}_T(\mu,\fm) \le \frac{2}{KT}\Ent_{\fm}(\mu). \]

\item \emph{(HWI inequality)}
Assume that $\mu=\rho\fm \in \cP(M)$ with $\rho \in H^1(M;\bH)$,
and let $\varphi$ be a semi-convex function deriving the optimal transport from $\mu$ to $\fm$
as in Theorem~$\ref{th:Monge}$.
Then we have
\[ \Ent_{\fm}(\mu) \le TI_{\fm}(\mu)
 +T \int_M \bL(\Nabla\varphi) \,d\mu -\frac{KT}{2}C^{\bH}_T(\mu,\fm). \]
\end{enumerate}
\end{theorem}

\begin{proof}
(i) Let $(\mu_t)_{t \in [0,T]}$ be the optimal transport from $\mu_0=\mu$ to $\mu_T=\fm$.
It follows from \eqref{eq:Kconv} and $\Ent_{\fm}(\fm)=0$ that, for $t \in [0,1)$,
\begin{equation}\label{eq:Tal}
\Ent_{\fm}(\mu_{tT}) \le (1-t)\Ent_{\fm}(\mu) -\frac{K}{2}(1-t)tT C^{\bH}_T(\mu,\fm).
\end{equation}
Since $\Ent_{\fm}(\mu_{tT}) \ge -\log(\fm[M])=0$,
dividing \eqref{eq:Tal} by $1-t$ and letting $t \to 1$ shows the claim.

(ii) We again use \eqref{eq:Tal} and find, for $t \in (0,1]$,
\[ \frac{\Ent_{\fm}(\mu_{tT}) -\Ent_{\fm}(\mu)}{t}
 \le -\Ent_{\fm}(\mu) -\frac{K}{2}(1-t) TC^{\bH}_T(\mu,\fm). \]
Proposition~\ref{pr:dEnt} yields
\begin{align*}
\liminf_{t \downarrow 0} \frac{\Ent_{\fm}(\mu_{tT}) -\Ent_{\fm}(\mu)}{tT}
&\ge \int_M \frac{d\rho}{\rho}(\Nabla \varphi) \,d\mu \\
&\ge  -\int_M \{ \bH(-d[\log\rho]) +\bL(\Nabla \varphi) \} \,d\mu \\
&= -I_{\fm}(\mu) -\int_M \bL(\Nabla\varphi) \,d\mu.
\end{align*}
This completes the proof.
$\qedd$
\end{proof}

In the Riemannian or Finsler case, the HWI inequality (with $T=1$) has the sharper form
\[ \Ent_{\fm}(\mu) \le 2\sqrt{I_{\fm}(\mu) C^{\bL}_1(\mu,\fm)} -\frac{K}{2}C^{\bH}_1(\mu,\fm) \]
and derives the \emph{log-Sobolev inequality}
\[ \Ent_{\fm}(\mu) \le \frac{2}{K}I_{\fm}(\mu) \]
by recalling $C^{\bH}_1=C^{\bL}_1$.
See \cite[\S 5]{OT1} and \cite[\S 6]{OT2} for related functional inequalities
on different kinds of entropies inspired by information theory.

\section{Heat flow}\label{sc:hf}%%%%%%%%%%%%%%%%%%%
%%%%%%%%%%%%%%%%%%%%%%

In this section, we study the evolution equation $\del_t u=\Lap^{\bH}_{\fm} u$
regarded as the \emph{heat equation} associated with
our nonlinear weighted Laplacian $\Lap^{\bH}_{\fm}$.
In the Euclidean setting, there are two ways to interpret heat flow as gradient flow.
The classical strategy is to consider heat flow as the gradient flow
of the Dirichlet energy in the $L^2$-space.
Another one initiated by the seminal work of Jordan et al.~\cite{JKO}
is to consider heat flow as the gradient flow of the relative entropy in the $L^2$-Wasserstein space.
These interpretations and their equivalence were generalized to various settings
(\cite{Sa}, \cite{Ogra}, \cite{GO}, \cite{Er}, \cite{OShf}, \cite{Maa}, \cite{GKO}, \cite{AGShf} etc.)
including singular spaces without differentiable structures.

We shall investigate these two strategies for our heat equation
and see that they produce different gradient flows,
because of the non-homogeneity of the Hamiltonian $\bH$.

Let $M$ be compact throughout the section,
and fix an auxiliary Riemannian metric $g$ of $M$ similarly to the previous section.

\subsection{Heat flow as gradient flow of energy}\label{ssc:hf}%%%%%%%%%%

For $(u_t)_{t \ge 0} \subset H^1(M;\bH) \cap H^1(M;\bL)$,
we readily see from \eqref{eq:btau*} that the energy form $\cE:=\cE_M$
defined in Subsection~\ref{ssc:harm} satisfies
\[ \del_{\ve}[\cE(u_t +\ve\phi)]|_{\ve=0} =\int_M d\phi(\Nabla u_t) \,d\fm
 =-\int_M \phi \Lap^{\bH}_{\fm} u_t \,d\fm \]
for every test function $\phi \in \cC^{\infty}(M)$
(provided that $\cE(u_t+\ve \phi)<\infty$ for some $\ve>0$).
Therefore, with respect to the $L^2$-inner product structure,
$\Lap^{\bH}_{\fm} u_t$ is the `gradient vector' of $-\cE$ at $u_t$
and the heat equation $\del_t u=\Lap^{\bH}_{\fm} u$ coincides with
the (descending) gradient flow equation of the potential function $\cE$.

Motivated by the above heuristic argument,
we shall build weak solutions to $\del_t u=\Lap^{\bH}_{\fm} u$
as gradient curves of $\cE$.
In fact, since $\cE$ is a convex and lower semi-continuous functional on $L^2(M;\fm)$
(Lemma~\ref{lm:lsc}), the classical theory due to Br\'ezis, Crandall, Liggett et al applies
(see \cite{Bre}, \cite{CL}).
For completeness, we repeat the construction in \cite[\S 3]{OShf} concerning the Finsler case
with the help of \cite{Ogra}.
Because of the author's familiarity, our construction follows the line of \cite{May}
which deals with the more general situation of convex functions on nonpositively curved metric spaces.

For $u \in H^1(M;\bH)$, define
\[ |\nabla(-\cE)|(u) :=\max\left\{ \limsup_{\bar{u} \to u}
 \frac{\cE(u)-\cE(\bar{u})}{\| \bar{u}-u \|_{L^2}},0 \right\}, \]
where $\bar{u} \in H^1(M;\bH)$ and $\bar{u} \to u$ is in $L^2(M;\fm)$.
By the convexity of $\cE$, we have $|\nabla(-\cE)|(u)=0$ if and only if $\cE(u)=0$,
i.e., $u$ is constant.
Similarly to \cite[Lemma~3.3]{OShf}, we can find a unique direction
attaining $|\nabla(-\cE)|(u)$.

\begin{lemma}\label{lm:Lgv}
If $0<|\nabla(-\cE)|(u)<\infty$, then there exists a unique element $w \in L^2(M;\fm)$
satisfying $\|w\|_{L^2}=|\nabla(-\cE)|(u)$ and
\begin{equation}\label{eq:Lgrad}
\limsup_{\bar{w} \to w} \lim_{t \downarrow 0}
 \frac{\cE(u)-\cE(u+t\bar{w})}{t\|\bar{w}\|_{L^2}} =|\nabla(-\cE)|(u),
\end{equation}
where the convergence $\bar{w} \to w$ is in $L^2(M;\fm)$.
\end{lemma}

\begin{proof}
Take a sequence $\{\hat{w}_i\}_{i \in \N} \subset L^2(M;\fm) \setminus \{0\}$ 
such that $\hat{w}_i \to 0$ and
\[ \lim_{i \to \infty} \frac{\cE(u)-\cE(u+\hat{w}_i)}{\|\hat{w}_i\|_{L^2}}
 =|\nabla(-\cE)|(u), \]
and normalize it as
\[ w_i:=\frac{|\nabla(-\cE)|(u)}{\|\hat{w}_i\|_{L^2}} \hat{w}_i. \]
It follows from the convexity of $\cE$ that
\[ \liminf_{i \to \infty} \lim_{t \downarrow 0} \frac{\cE(u)-\cE(u+tw_i)}{t\|w_i\|_{L^2}}
 \ge \lim_{i \to \infty} \frac{\cE(u)-\cE(u+\hat{w}_i)}{\|\hat{w}_i\|_{L^2}}
 =|\nabla(-\cE)|(u). \]
Thus we have, together with the definition of $|\nabla(-\cE)|(u)$,
\begin{equation}\label{eq:Lgv}
\lim_{i \to \infty} \lim_{t \downarrow 0} \frac{\cE(u)-\cE(u+tw_i)}{t\|w_i\|_{L^2}}
 =|\nabla(-\cE)|(u).
\end{equation}
The convexity of $\cE$ also implies that, for any $i,j \ge 1$,
\begin{align*}
&\lim_{t \downarrow 0} \frac{\cE(u)-\cE(u+t(w_i+w_j)/2)}{t\|(w_i+w_j)/2\|_{L^2}} \\
&\ge \frac{|\nabla(-\cE)|(u)}{\|(w_i+w_j)/2\|_{L^2}}
 \lim_{t \downarrow 0} \left\{ \frac{\cE(u)-\cE(u+tw_i)}{2t\|w_i\|_{L^2}}
 +\frac{\cE(u)-\cE(u+tw_j)}{2t\|w_j\|_{L^2}} \right\},
\end{align*}
we used $\|w_i\|_{L^2}=\|w_j\|_{L^2}=|\nabla(-\cE)|(u)$.
Since the LHS is not greater than $|\nabla(-\cE)|(u)$, we find from \eqref{eq:Lgv}
that $\liminf_{i,j \to \infty}\|w_i+w_j\|_{L^2} \ge 2|\nabla(-\cE)|(u)$.
Combining this with $\|w_i\|_{L^2}=|\nabla(-\cE)|(u)$ for all $i$
implies that $\{w_i\}_{i \in \N}$ is a Cauchy sequence in $L^2(M;\fm)$
and converges to some element $w \in L^2(M;\fm)$.
The desired equation \eqref{eq:Lgrad} holds due to \eqref{eq:Lgv}.

The uniqueness follows from the strict convexity of $\|\cdot\|_{L^2}$
and the convexity of $\cE$.
$\qedd$
\end{proof}

We excluded the case of $|\nabla(-\cE)|(u)=0$ since \eqref{eq:Lgrad}
does not necessarily hold with $w=0$.

\begin{definition}[Gradient vectors of $-\cE$]\label{df:Lgv}
At $u \in H^1(M;\bH)$ with $0<|\nabla(-\cE)|(u)<\infty$,
we define the \emph{gradient vector} of $-\cE$ by $\nabla(-\cE)(u):=w \in L^2(M;\fm)$,
where $w$ is given in Lemma~\ref{lm:Lgv}.
We also set $\nabla(-\cE)(u):=0 \in L^2(M;\fm)$ if $|\nabla(-\cE)|(u)=0$.
\end{definition}

A gradient curve of $\cE$ should be understood as a solution to $\del_t u=\nabla(-\cE)(u)$.
We explain how to formulate and construct it.
Fix $u_0 \in H^1(M;\bH)$ and $\delta>0$.
Denote by $\sU_{\delta}(u_0) \in H^1(M;\bH)$ the unique minimizer of
the strictly convex, lower semi-continuous functional
\begin{equation}\label{eq:U(u)}
u \ \longmapsto \ \cE(u)+\frac{\|u-u_0\|_{L^2}^2}{2\delta}.
\end{equation}
Then, as $k \to \infty$, the discrete approximation scheme $(\sU_{t/k})^k(u_0)$
converges to a curve $(u_t)_{t>0} \subset L^2(M;\fm)$
such that $\lim_{t \downarrow 0} u_t=u_0$ in $L^2(M;\fm)$
and $\cE(u_t) \le \cE(u_0)$ for all $t>0$ (we refer to \cite[Theorem~1.13]{May} for details).
Moreover, the convergence is uniform on $(0,T]$ for each $T>0$,
and $(u_t)_{t \ge 0}$ satisfies the following:
\begin{itemize}
\item The curve $(u_t)_{t>0}$ is locally Lipschitz in $L^2(M;\fm)$ and satisfies
\begin{equation}\label{eq:Lspe}
\lim_{\delta \downarrow 0} \frac{\|u_{t+\delta}-u_t\|_{L^2}}{\delta} =|\nabla(-\cE)|(u_t)
\end{equation}
at all $t \ge 0$ (\cite[Theorems~2.9, 2.17]{May}).
In particular, $|\nabla(-\cE)|(u_t)<\infty$ holds and
the gradient vector $\nabla(-\cE)(u_t) \in L^2(M;\fm)$ is well-defined for all $t>0$.

\item The function $t \longmapsto \cE(u_t)$ is continuous on $[0,\infty)$
and locally Lipschitz on $(0,\infty)$ (\cite[Corollary~2.10]{May}).
Combining \eqref{eq:Lspe} with \cite[Lemma~2.11, Theorem~2.14]{May},
we have, for all $t \ge 0$,
\begin{equation}\label{eq:Espe}
\lim_{\delta \downarrow 0} \frac{\cE(u_t)-\cE(u_{t+\delta})}{\delta} =|\nabla(-\cE)|(u_t)^2.
\end{equation}
\end{itemize}

With the help of \eqref{eq:Lspe} and \eqref{eq:Espe},
the same discussion as Lemma~\ref{lm:Lgv}
(replacing $\hat{w}_i$ and $w_i$ with $u_{t+\delta}-u_t$ and $(u_{t+\delta}-u_t)/\delta$,
respectively) shows that
\begin{equation}\label{eq:dtu}
\lim_{\delta \downarrow 0} \left\| \frac{u_{t+\delta}-u_t}{\delta} -\nabla(-\cE)(u_t) \right\|_{L^2} =0
\end{equation}
for all $t>0$, i.e., $\del_t u_t=\nabla(-\cE)(u_t)$ in the weak sense.
In addition, we obtain the following similarly to \cite[Lemma~6.4]{Ogra}.

\begin{lemma}\label{lm:U(u)}
\begin{enumerate}[{\rm (i)}]
\item Fix arbitrary $t>0$ and put $\sigma_{\delta}:=\sU_{\delta}(u_t)$
by suppressing the dependence on $t$.
Then we have
\[ \lim_{\delta \downarrow 0} \frac{\|u_{t+\delta}-\sigma_{\delta}\|_{L^2}}{\delta}=0. \]
\item
Moreover, $d\sigma_{\delta}$ converges to $du_t$ as $\delta \downarrow 0$
in the $L^2$-norm with respect to $g$.
\end{enumerate}
\end{lemma}

\begin{proof}
(i)
Recall that $|\nabla(-\cE)|(u_t)<\infty$.
If $|\nabla(-\cE)|(u_t)=0$, then $\cE(u_t)=0$ and hence $\sigma_{\delta}=u_t$ for all $\delta>0$.
Thus the claim follows from \eqref{eq:Lspe}.

Suppose $|\nabla(-\cE)|(u_t)>0$.
By \eqref{eq:dtu}, the claim is equivalent to
\[ \lim_{\delta \downarrow 0}
 \left\| \frac{\sigma_{\delta}-u_t}{\delta}-\nabla(-\cE)(u_t) \right\|_{L^2} =0. \]
Take a sequence
$\{w_i\}_{i \in \N} \subset L^2(M;\fm)$ as in the proof of Lemma~\ref{lm:Lgv},
namely $\|w_i\|_{L^2}=|\nabla(-\cE)|(u_t)$ for all $i$ and
\[ \lim_{i \to \infty} \lim_{s \downarrow 0} \frac{\cE(u_t)-\cE(u_t +sw_i)}{s\|w_i\|_{L^2}}
 =|\nabla(-\cE)|(u_t). \]
Put
\[ w_i^{\delta}:=u_t+\frac{\|\sigma_{\delta}-u_t\|_{L^2}}{\|w_i\|_{L^2}}w_i \]
and observe from the choice of $\sigma_{\delta}$ that
\[ \liminf_{\delta \downarrow 0} \frac{\cE(u_t)-\cE(\sigma_{\delta})}{\|\sigma_{\delta}-u_t\|_{L^2}}
 \ge \lim_{\delta \downarrow 0} \frac{\cE(u_t)-\cE(w_i^{\delta})}{\|\sigma_{\delta}-u_t\|_{L^2}}
 =\lim_{s \downarrow 0} \frac{\cE(u_t)-\cE(u_t +sw_i)}{s\|w_i\|_{L^2}}. \]
Letting $i \to \infty$ implies
\[ \liminf_{\delta \downarrow 0} \frac{\cE(u_t)-\cE(\sigma_{\delta})}{\|\sigma_{\delta}-u_t\|_{L^2}}
 \ge |\nabla(-\cE)|(u_t), \]
and hence
\begin{equation}\label{eq:dsigma}
\lim_{\delta \downarrow 0} \frac{\cE(u_t)-\cE(\sigma_{\delta})}{\|\sigma_{\delta}-u_t\|_{L^2}}
 =|\nabla(-\cE)|(u_t).
\end{equation}
If we have $\lim_{\delta \downarrow 0}\|\sigma_{\delta}-u_t\|_{L^2}/\delta =|\nabla(-\cE)|(u_t)$,
then we can complete the proof along the same line as Lemma~\ref{lm:Lgv}.
To see this, we observe
\[ \cE(\sigma_{\delta}) +\frac{\|\sigma_{\delta}-u_t\|_{L^2}^2}{2\delta}
 \le \cE(u_t+\delta w_i) +\frac{\delta^2 |\nabla(-\cE)|(u_t)^2}{2\delta} \]
by the choice of $\sigma_{\delta}$, and
\[ \cE(\sigma_{\delta}) \ge \cE(u_t) -|\nabla(-\cE)|(u_t) \|\sigma_{\delta}-u_t\|_{L^2} \]
by the convexity of $\cE$.
Thus we find
\begin{align*}
\frac{\|\sigma_{\delta}-u_t\|_{L^2}^2}{2\delta^2}
&\le \frac{\cE(u_t+\delta w_i)-\cE(\sigma_{\delta})}{\delta}
 +\frac{|\nabla(-\cE)|(u_t)^2}{2} \\
&\le \frac{\cE(u_t+\delta w_i)-\cE(u_t)}{\delta}
 +\frac{\|\sigma_{\delta}-u_t\|_{L^2}}{\delta} |\nabla(-\cE)|(u_t)
 +\frac{|\nabla(-\cE)|(u_t)^2}{2}.
\end{align*}
Letting $\delta \downarrow 0$ and then $i \to \infty$, we obtain
\[ \left( \limsup_{\delta \downarrow 0} \frac{\|\sigma_{\delta}-u_t\|_{L^2}}{\delta} \right)^2
 \le 2|\nabla(-\cE)|(u_t) \cdot \limsup_{\delta \downarrow 0} \frac{\|\sigma_{\delta}-u_t\|_{L^2}}{\delta}
 -|\nabla(-\cE)|(u_t)^2. \]
Therefore $\limsup_{\delta \downarrow 0} \|\sigma_{\delta}-u_t\|_{L^2}/\delta=|\nabla(-\cE)|(u_t)$
and similarly $\liminf_{\delta \downarrow 0} \|\sigma_{\delta}-u_t\|_{L^2}/\delta=|\nabla(-\cE)|(u_t)$.

(ii)
The strict convexity of $\bH$ and \eqref{eq:dsigma} show that $\{d\sigma_{\delta}\}_{\delta>0}$
is a Cauchy sequence and converges to $du_t$, i.e.,
$\lim_{\delta \downarrow 0}\int_M |d\sigma_{\delta}-du_t|_g^2 \,d\fm=0$.
(We remark that, however, $\int_M |d\sigma_{\delta}|_g^2 \,d\fm=\infty$ may happen.)
$\qedd$
\end{proof}

We are ready to show the main result of the subsection.
Due to technical difficulties, we impose the following assumption:
\begin{itemize}
\item[{\bf (A)}]
For each $\ve>0$, there is $C_{\ve}>0$ such that
$|\btau^*(\alpha)-\btau^*(\beta)|_g \le C_{\ve}|\alpha-\beta|_g$
for all $\alpha,\beta \in T_x^*M$, $x \in M$, with $|\btau^*(\alpha)-\btau^*(\beta)|_g \ge \ve$.
\end{itemize}
This is obviously a strong constraint since it is global in the vertical direction,
however, we need this kind of condition to control the behavior of $\Nabla \sigma_{\delta}$
by that of $d\sigma_{\delta}$.

\begin{theorem}[Heat flow as gradient flow of energy]\label{th:Lgf}
Suppose {\bf (A)}.
Given an initial function $u_0 \in H^1(M;\bH)$,
let $(u_t)_{t \ge 0} \subset H^1(M;\bH)$ be a gradient curve of $\cE$ constructed as above.
Then $(u_t)_{t \ge 0}$ solves the heat equation $\del_t u_t=\Lap^{\bH}_{\fm}u_t$
in the weak sense that
\begin{equation}\label{eq:iwhf}
\int_0^{\infty} \!\int_M \del_t \phi \cdot u_t \,d\fm dt
 =\int_0^{\infty} \!\int_M d\phi_t (\Nabla u_t) \,d\fm dt
\end{equation}
for all $\phi \in \cC_c^{\infty}((0,\infty) \times M)$,
where $\phi_t:=\phi(t,\cdot)$.

Moreover, for all $t>0$, the distributional Laplacian $\Lap^{\bH}_{\fm} u_t$ is absolutely continuous
with respect to $\fm$ and its density function coincides with $\nabla(-\cE)(u_t)$.
In particular, we have
\[ \|\Lap^{\bH}_{\fm} u_t\|_{L^2} =\|\nabla(-\cE)(u_t)\|_{L^2} =|\nabla(-\cE)|(u_t) \]
for all $t>0$.
\end{theorem}

\begin{proof}
We fix a test function $\phi \in \cC_c^{\infty}((0,\infty) \times M)$ and shall show that
\begin{equation}\label{eq:whf}
\lim_{\delta \downarrow 0} \frac{1}{\delta} \int_M (\phi_{t+\delta} u_{t+\delta}-\phi_t u_t) \,d\fm
 =\int_M \{ \del_t \phi_t \cdot u_t -d\phi_t(\Nabla u_t) \} \,d\fm
\end{equation}
for all $t>0$, which implies \eqref{eq:iwhf} by integration.
We remark that the RHS is well-defined because
\[ |d\phi_t(\Nabla u_t)| \le |d\phi_t|_g |\Nabla u_t|_g \le |d\phi_t|_g (\ve+C_{\ve}|du_t|_g) \]
by {\bf (A)} and
\[ \int_M |du_t|_g \,d\fm \le \int_M \{ C_1 \bH(du_t)+C_2 \} \,d\fm <\infty \]
for some $C_1,C_2>0$ (super-linearity).

Fix $t>0$.
For small $\delta,\ve>0$, let us compare $\sigma_{\delta}:=\sU_{\delta}(u_t)$
and $\sigma_{\delta}^{\ve}:=\sigma_{\delta}+\ve\phi_t$.
By the definition of $\sU_{\delta}(u_t)$, we have
\[ \cE(\sigma_{\delta}) +\frac{\|\sigma_{\delta}-u_t\|_{L^2}^2}{2\delta}
 \le \cE(\sigma_{\delta}^{\ve}) +\frac{\|\sigma_{\delta}^{\ve}-u_t\|_{L^2}^2}{2\delta}. \]
On the one hand, by expanding $\|\sigma_{\delta}-u_t\|_{L^2}^2$
and $\|\sigma_{\delta}^{\ve}-u_t\|_{L^2}^2$,
we find
\[ \lim_{\ve \downarrow 0}
 \frac{\|\sigma_{\delta}^{\ve}-u_t\|_{L^2}^2 -\|\sigma_{\delta}-u_t\|_{L^2}^2}{2\ve}
 =\int_M \phi_t(\sigma_{\delta}-u_t) \,d\fm. \]
On the other hand, we deduce from the convexity of $\bH$ that
\[ \lim_{\ve \downarrow 0}
 \frac{\cE(\sigma_{\delta}^{\ve}) -\cE(\sigma_{\delta})}{\ve}
 =\int_M \lim_{\ve \downarrow 0} \frac{\bH(d\sigma_{\delta}^{\ve})-\bH(d\sigma_{\delta})}{\ve} \,d\fm
 =\int_M d\phi_t(\Nabla \sigma_{\delta}) \,d\fm. \]
Therefore we have, with the help of Lemma~\ref{lm:U(u)}(ii) and {\bf (A)},
\[ \liminf_{\delta \downarrow 0} \frac{1}{\delta}
 \int_M \phi_t(\sigma_{\delta}-u_t) \,d\fm
 \ge -\lim_{\delta \downarrow 0} \int_M d\phi_t(\Nabla \sigma_{\delta}) \,d\fm
 =-\int_M d\phi_t(\Nabla u_t) \,d\fm. \]
To be precise, for each $\ve>0$, we have by {\bf (A)}
\[ \left| \int_M d\phi_t(\Nabla \sigma_{\delta} -\Nabla u_t) \,d\fm \right|
 \le \int_M |d\phi_t|_g (\ve+C_{\ve}|d\sigma_{\delta}-du_t|_g) \,d\fm \]
which tends to $0$ as $\delta \downarrow 0$ and then $\ve \downarrow 0$.
Together with Lemma~\ref{lm:U(u)}(i), we obtain
\begin{align*}
\liminf_{\delta \downarrow 0} \frac{1}{\delta} \int_M (\phi_{t+\delta} u_{t+\delta} -\phi_t u_t) \,d\fm
&= \liminf_{\delta \downarrow 0} \frac{1}{\delta}
 \int_M \{ (\phi_{t+\delta}-\phi_t) u_{t+\delta} +\phi_t (u_{t+\delta}-u_t) \} \,d\fm \\
&= \int_M \del_t \phi_t \cdot u_t \,d\fm
 +\liminf_{\delta \downarrow 0} \frac{1}{\delta} \int_M \phi_t(\sigma_{\delta}-u_t) \,d\fm \\
&\ge \int_M \{ \del_t \phi_t \cdot u_t -d\phi_t(\Nabla u_t) \} \,d\fm.
\end{align*}
By exchanging $\phi$ with $-\phi$, we obtain the desired equation \eqref{eq:whf}.

For a time-independent test function $\phi \in \cC^{\infty}(M)$,
\eqref{eq:dtu} and \eqref{eq:whf} show that
\[ \int_M \phi \nabla(-\cE)(u_t) \,d\fm
 =\lim_{\delta \downarrow 0} \frac{1}{\delta} \int_M \phi (u_{t+\delta}-u_t) \,d\fm
 =-\int_M d\phi(\Nabla u_t) \,d\fm \]
for all $t>0$.
Therefore the distributional Laplacian $\Lap^{\bH}_{\fm} u_t$ is absolutely continuous
with respect to $\fm$ and the density function is $\nabla(-\cE)(u_t)$.
$\qedd$
\end{proof}

Given two solutions $(u_t)_{t \ge 0}$, $(\bar{u}_t)_{t \ge 0}$ to the heat equation
constructed in Theorem~\ref{th:Lgf},
applying \eqref{eq:whf} with $\phi=u-\bar{u}$ (this is possible by \cite[Remark~3.2(i)]{OShf}),
we find
\begin{align*}
\frac{d}{dt}\left[ \| u_t-\bar{u}_t \|_{L^2}^2 \right]
&= \frac{d}{dt} \left[ \int_M \{ (u_t-\bar{u}_t)u_t -(u_t-\bar{u}_t)\bar{u}_t \} \,d\fm \right] \\
&= \int_M \{ \del_t(u_t-\bar{u}_t)(u_t-\bar{u}_t) -d(u_t-\bar{u}_t)(\Nabla u_t -\Nabla \bar{u}_t) \} \,d\fm \\
&= -2 \int_M (du_t-d\bar{u}_t)(\Nabla u_t -\Nabla \bar{u}_t) \,d\fm
\end{align*}
at almost every $t>0$.
The convexity of $\bH$ implies
\[ (du_t-d\bar{u}_t)(\Nabla u_t -\Nabla \bar{u}_t)
 =(du_t-d\bar{u}_t) \left( \sum_{i=1}^n \bH_{\alpha_i}(du_t) \del_{x^i}
 -\sum_{i=1}^n \bH_{\alpha_i}(d\bar{u}_t) \del_{x^i} \right)
 \ge 0, \]
which yields the contraction $(\|u_t-\bar{u}_t\|_{L^2}^2)' \le 0$ of the heat flow.
In particular, if $u_0=\bar{u}_0$ $\fm$-almost everywhere,
then $u_t=\bar{u}_t$ $\fm$-almost everywhere for all $t>0$.

\subsection{Gradient flow of the relative entropy}\label{ssc:Wgf}%%%%%%%%%

We next consider the gradient flow of the relative entropy $\Ent_{\fm}$
with respect to the Lagrangian cost function $C^{\bL}_T$ introduced in Section~\ref{sc:opt}.
We will verify that such a flow is produced from weak solutions to the evolution equation
\begin{equation}\label{eq:Ent-gf}
\del_t u =-\div_{\fm}(u \Nabla[-\log u]),
\end{equation}
which is different from the heat equation $\del_t u=\Lap^{\bH}_{\fm}u$
due to the non-homogeneity of the gradient operator $\Nabla$.

Let us begin with a discussion on how to define gradient flows.
Our strategy follows the metric approach recently intensively studied by
Ambrosio, Gigli, Savar\'e and others (see \cite{AGSbook}).
Given a $\cC^1$-function $f:M \lra \R$,
the \emph{gradient flow} of $f$ with respect to $\bL$ (or $\bH$) is introduced
as the family of maps $\{ \cG_t \}_{t \ge 0}$ sending $x$ to $\cG_t(x)=\eta(t)$,
where $\eta$ is the $\cC^1$-curve solving
\begin{equation}\label{eq:Hgf}
\dot{\eta}(t) =\Nabla (-f)\big( \eta(t) \big), \qquad \eta(0)=x.
\end{equation}
Since $v=\Nabla(-f)(\eta(t))$ is the unique element satisfying
\[ -df_{\eta(t)}(v) =\bH(-df_{\eta(t)}) +\bL(v) \]
in $T_{\eta(t)}M$ and $-df_{\eta(t)}(v) \le \bH(-df_{\eta(t)}) +\bL(v)$ holds in general,
the gradient flow equation \eqref{eq:Hgf} can be rewritten as
\[ (f \circ \eta)'(t) \le -\bH(-df_{\eta(t)}) -\bL\big( \dot{\eta}(t) \big). \]
Integrating this inequality leads to
\[ f\big( \eta(t) \big) \le f\big( \eta(s) \big)
 -\int_s^t \left\{ \bH(-df_{\eta(r)}) +\bL\big( \dot{\eta}(r) \big) \right\} dr \]
for all $0 \le s<t$.
In order to avoid the use of the derivative $df$, we observe that, at any $y \in M$,
\[ \bH(-df_y) =\sup_{v \in T_yM}\{ -df_y(v)-\bL(v) \}
 =\sup_{\gamma} \left\{ \lim_{t \downarrow 0}
 \frac{f(y)-f(\gamma(t))}{t}-\bL\big( \dot{\gamma}(0) \big) \right\}, \]
where $\gamma$ runs over all differentiable curves with $\gamma(0)=y$.
Note finally that
\[ \bL\big( \dot{\eta}(r) \big)
 =\lim_{\delta \downarrow 0}\frac{c^{\bL}_{\delta}(\eta(r),\eta(r+\delta))}{\delta}. \]
Then, it is natural to introduce the following metric definition.

\begin{definition}[Gradient flow of $\Ent_{\fm}$]\label{df:Wgf}
A curve $(\mu_t)_{t \ge 0} \subset \cP(M)$ which is continuous in the weak topology of $\cP(M)$
is called a \emph{gradient curve} of $\Ent_{\fm}$ if $\Ent_{\fm}(\mu_t)<\infty$ for all $t>0$ and if
\begin{equation}\label{eq:Wgf}
\Ent_{\fm}(\mu_t) \le \Ent_{\fm}(\mu_s) -\int_s^t
 \{ |d(-\Ent_{\fm})|_{\bH}(\mu_r) +|\dot{\mu}_r|_{\bL} \} \,dr
\end{equation}
holds for all $0 \le s<t$.
Here we set
\begin{align*}
|\dot{\mu}_r|_{\bL}
&:= \limsup_{\delta \downarrow 0}\frac{C^{\bL}_{\delta}(\mu_r,\mu_{r+\delta})}{\delta}, \\
|d(-\Ent_{\fm})|_{\bH}(\nu)
&:= \sup_{(\nu_s)_{s \in [0,T]}} \left\{ \limsup_{s \downarrow 0}
 \frac{\Ent_{\fm}(\nu)-\Ent_{\fm}(\nu_s)}{s}-|\dot{\nu}_0|_{\bL} \right\}
\end{align*}
for $\nu \in \cP(M)$ with $\Ent_{\fm}(\nu)<\infty$,
where $(\nu_s)_{s \in [0,T]}$ runs over all optimal transports with $\nu_0=\nu$
(possibly $\nu_t \equiv \nu$, so that $|d(-\Ent_{\fm})|_{\bH} \ge 0$).
\end{definition}

We assume that $\Ent_{\fm}$ is $K$-convex in the sense of \eqref{eq:Kconv}
for some $K \in \R$ and all $T>0$,
and shall show that, for a weak solution $(\rho_t)_{t \ge 0}$ to \eqref{eq:Ent-gf}
with $\mu_t:=\rho_t \fm \in \cP(M)$, the curve $(\mu_t)_{t \ge 0}$ enjoys \eqref{eq:Wgf}.
For simplicity, suppose that the following condition is satisfied:
\begin{itemize}
\item[{\bf (C)}]
$\rho_t$ is locally $g$-Lipschitz on $(0,\infty) \times M$,
$\lim_{r \downarrow 0}d\rho_{t+r}=d\rho_t$ in the $L^2$-norm with respect to $g$
and $\inf_M \rho_t>0$ for every $t>0$
(then clearly $\Ent_{\fm}(\mu_t)<\infty$ for $t>0$).
\end{itemize}
Observe that the heat flow constructed in the previous subsection satisfies
$\lim_{r \downarrow 0}d\rho_{t+r}=d\rho_t$ similarly to Lemma~\ref{lm:U(u)}(ii).

First of all, it follows from Proposition~\ref{pr:dEnt} that
\begin{equation}\label{eq:Wgf1}
|d(-\Ent_{\fm})|_{\bH}(\mu) \le \int_M \bH\left( -\frac{d\rho}{\rho} \right) d\mu.
\end{equation}
Next, observe from \eqref{eq:Ent-gf} (with the test function $\log \rho_t$) that
\begin{equation}\label{eq:Wgf2}
\frac{d}{dt} [\Ent_{\fm}(\mu_t)] =\int_M d(\log\rho_t) (\rho_t \Nabla[-\log\rho_t]) \,d\fm
 =\int_M \frac{d\rho_t}{\rho_t} (\Nabla[-\log\rho_t]) \,d\mu_t.
\end{equation}
Finally, $|\dot{\mu}_t|_{\bL}$ is estimated similarly to \cite[Proposition~3.7]{GKO} as follows.

\begin{lemma}\label{lm:Wgf3}
At each $t>0$, it holds that
\[ |\dot{\mu}_t|_{\bL} \le \int_M \bL(\Nabla[-\log \rho_t]) \,d\mu_t. \]
\end{lemma}

\begin{proof}
For $s>0$, we have by \cite[Proposition~21]{BeBu}
\[ C^{\bL}_s(\mu_t,\mu_{t+s}) =\sup_{\varphi}
 \left\{ \int_M \varphi_{t+s} \,d\mu_{t+s} -\int_M \varphi_t \,d\mu_t \right\}, \]
where the supremum is taken over all $g$-Lipschitz functions
$\varphi:[t,t+s] \times M \lra \R$ which are $\cC^1$ on $(t,t+s) \times M$ and satisfy
the inequality (of Hamilton--Jacobi type)
\[ \del_r \varphi +\bH(d\varphi) \le 0 \quad \text{on}\ (t,t+s) \times M. \]
Together with \eqref{eq:Ent-gf}, we obtain
\begin{align*}
\int_M \varphi_{t+s} \,d\mu_{t+s} -\int_M \varphi_t \,d\mu_t
&= \int_M \int_0^s \del_r [\varphi_{t+r} \rho_{t+r}] \,dr d\fm \\
&\le \int_0^s \int_M \{ -\bH(d\varphi_{t+r})
 +d\varphi_{t+r}(\Nabla[-\log \rho_{t+r}]) \} \,d\mu_{t+r} dr \\
&\le \int_0^s \int_M \bL(\Nabla[-\log \rho_{t+r}]) \,d\mu_{t+r} dr.
\end{align*}
Therefore we conclude that
\[ |\dot{\mu}_t|_{\bL} \le \lim_{s \downarrow 0} \frac{1}{s}
 \int_0^s \int_M \bL(\Nabla[-\log \rho_{t+r}]) \,d\mu_{t+r} dr
 =\int_M \bL(\Nabla[-\log \rho_t]) \,d\mu_t \]
by the bounded convergence theorem
(since $\rho_t$ is locally $g$-Lipschitz and $\inf_M \rho_t>0$)
and $\lim_{r \downarrow 0}d\rho_{t+r}=d\rho_t$.
$\qedd$
\end{proof}

Combining above three estimates, we immediately obtain the following.

\begin{theorem}[Gradient flow of $\Ent_{\fm}$]\label{th:Wgf}
Assume that $\Ent_{\fm}$ is $K$-convex in the sense of \eqref{eq:Kconv}
for some $K \in \R$ and all $T>0$.
Let $(\rho_t)_{t \ge 0}$ be a weak solution to the equation \eqref{eq:Ent-gf} satisfying
{\bf (C)} and $\mu_t:=\rho_t \fm \in \cP(M)$ for all $t \ge 0$.
Then $(\mu_t)_{t \ge 0}$ is a gradient curve of $\Ent_{\fm}$ in the sense of \eqref{eq:Wgf}.
\end{theorem}

\begin{proof}
Since $d\rho_t/\rho_t =d[\log \rho_t]$, we deduce from \eqref{eq:Wgf2}, \eqref{eq:Wgf1}
and Lemma~\ref{lm:Wgf3} that
\begin{align*}
\frac{d}{dt} [\Ent_{\fm}(\mu_t)]
&=-\int_M \left\{ \bH(-d[\log \rho_t]) +\bL(\Nabla[-\log\rho_t]) \right\} d\mu_t \\
&\le -|d(-\Ent_{\fm})|_{\bH}(\mu_t) -|\dot{\mu}_t|_{\bL}.
\end{align*}
This completes the proof.
$\qedd$
\end{proof}

\begin{example}\label{ex:p-gf}
For the $p$-homogeneous deformation $\bH=|\cdot|_g^p/p$ of a Riemannian metric $g$,
the equation \eqref{eq:Ent-gf} turns out
\[ \del_t u =-\div_{\fm}(u |\Nabla^g[-\log u]|_g^{p-2} \cdot \Nabla^g[-\log u])
 =\div_{\fm}\left( \left| \frac{\Nabla^g u}{u} \right|_g^{p-2} \cdot \Nabla^g u \right). \]
This recovers the case of $\bL(v)=|v|^q/q$ on $\R^n$ studied in \cite[\S 11.3]{AGSbook} and \cite{Agu},
in which more general entropy functionals are investigated.
For instance, employing the \emph{R\'enyi--Tsallis entropy}
\[ S_k(\rho\fm):=\frac{1}{k(k-1)} \int_M \rho^k \,d\fm, \qquad
 k=\frac{2p-3}{p-1}=3-q \]
leads to the \emph{$p$-Laplacian equation}:
$\del_t u=\Lap^{\bH}_{\fm}u=\div_{\fm}(|\Nabla^g u|_g^{p-2} \cdot \Nabla^g u)$.
Indeed, Proposition~\ref{pr:dEnt} is read as
\[ \liminf_{t \downarrow 0} \frac{S_k(\mu_t)-S_k(\mu)}{t}
 \ge \int_M \rho^{k-2} d\rho(\Nabla\varphi) \,d\mu, \]
and then \eqref{eq:Wgf1}, \eqref{eq:Wgf2} and Lemma~\ref{lm:Wgf3} become
\begin{align*}
|d(-S_k)|_{\bH}(\mu)
 &\le \int_M \bH(-\rho^{k-2} d\rho) \,d\mu
 =\frac{1}{p} \int_M |d\rho|_g^p \rho^{(k-2)p} \,d\mu, \\
\frac{d}{dt}[S_k(\mu_t)]
 &= -\int_M d\rho_t(\Nabla \rho_t) \rho_t^{k-2} \,d\fm
 =-\int_M \{ \bH(d\rho_t)+\bL(\Nabla \rho_t) \} \rho_t^{k-2} \,d\fm, \\
|\dot{\mu}_t|_{\bL}
 &\le \int_M \bL\left( -\frac{\Nabla \rho_t}{\rho_t} \right) d\mu_t
 =\frac{1}{q} \int_M |\Nabla^g \rho_t|_g^p \rho_t^{-q} \,d\mu_t
\end{align*}
(compare the last inequality with \cite[Lemma~7.2]{AGSden}).
Since $(k-2)p=-q$, we find as desired
\[ \frac{d}{dt}[S_k(\mu_t)]
 =-\int_M \left( \frac{|d\rho_t|_g^p}{p} +\frac{|\Nabla^g \rho_t|_g^p}{q} \right) \rho_t^{-q} \,d\mu_t
 \le -|d(-S_k)|_{\bH}(\mu_t) -|\dot{\mu}_t|_{\bL}. \]
\end{example}

In the Finsler case, the equation \eqref{eq:Ent-gf} coincides with
the heat equation with respect to the reverse Finsler structure $\rev{F}(v):=F(-v)$,
$\del_t \rho =-\Lap^F_{\fm}(-\rho)$, as shown in \cite{OShf}.
In general, it is unclear if one can regard \eqref{eq:Ent-gf}
as the heat equation of some Hamiltonian.

%%%%%%%%%%%%%%%%%%%%%%%%%%%%%%%%%%%%%%
\renewcommand{\thesection}{\Alph{section}}
\setcounter{section}{0}
\section{Appendix}\label{sc:apdx}%%%%%%%%

In this appendix, for self-containedness,
we give a way of constructing canonical frames as well as a detailed calculation
of the curvature operator in coordinates.
In both subjects, we follow the line of \cite{Le2} with proofs.

\subsection{How to construct canonical frames}\label{ssc:cano}%%%%%%%%%%

Fix $\alpha \in T_x^*M \setminus \{0\}$ and put $\bm{\alpha}(t):=\Phi_t(\alpha)$
for $t \in (-\ve,\ve)$ throughout the appendix.
Our starting ingredient is a family of smooth curves
$\bm{\xi}_i(t) =\sum_{j=1}^n \xi_i^j(t) \del_{\alpha_j} \in \cV_{\bm{\alpha}(t)}$, $i=1,\ldots,n$,
orthonormal with respect to the inner product introduced in Lemma~\ref{lm:quad}, namely
\[ \sum_{k,l=1}^n \bH_{\alpha_k \alpha_l}\big( \bm{\alpha}(t) \big)
 \xi_i^k(t) \xi_j^l(t) =\delta_{ij} \qquad \text{for all}\ i,j\ \text{and}\ t. \]
We first show that the curves $\bar{e}^t_i :=(d\Phi_t)^{-1}(\bm{\xi}_i(t)) \in \cJ_{\alpha}^t$
satisfy the first condition in \eqref{eq:cano}.

\begin{lemma}\label{lm:orth}
We have $\omega(\dot{\bar{e}}^t_i,\bar{e}^t_j)=\delta_{ij}$ for all $i,j$ and $t$.
\end{lemma}

\begin{proof}
Note that, by $\dot{\Phi}_t=\ora{\bH} \circ \Phi_t$,
\begin{align*}
\dot{\bar{e}}^t_i
&=(d\Phi_t)^{-1}\left( \sum_{k=1}^n \dot{\xi}_i^k(t) \del_{\alpha_k} \right) \\
&\quad
 -(d\Phi_t)^{-1} \left[ \sum_{k,l=1}^n \xi_i^l(t) \left\{ \bH_{\alpha_k \alpha_l}\big( \bm{\alpha}(t) \big) \del_{x^k}
 -\bH_{x^k \alpha_l}\big( \bm{\alpha}(t) \big) \del_{\alpha_k} \right\} \right].
\end{align*}
Therefore we have, by \eqref{eq:omega},
\begin{align*}
\omega(\dot{\bar{e}}^t_i,\bar{e}^t_j)
&=-\omega\left( \sum_{k,l=1}^n
 \bH_{\alpha_k \alpha_l}\big( \bm{\alpha}(t) \big) \xi_i^l(t) \del_{x^k}, \bm{\xi}_j(t) \right)
 =\sum_{k,l=1}^n \bH_{\alpha_k \alpha_l}\big( \bm{\alpha}(t) \big) \xi_i^l(t) \xi_j^k(t) \\
&=\delta_{ij}.
\end{align*}
$\qedd$
\end{proof}

One can construct a canonical frame from an orthonormal frame as follows.

\begin{proposition}{\rm (see \cite[Proposition~3.2]{Le2})}\label{pr:O^t}
Given a frame $\ol{E}^t=(\bar{e}^t_i)_{i=1}^n$ orthonormal in the sense of Lemma~{\rm \ref{lm:orth}},
put $\Omega^t_{ij}:=\omega(\dot{\bar{e}}^t_i,\dot{\bar{e}}^t_j)$
and let $O^t$ be the solution to
\[ \dot{O}^t=\frac{1}{2} O^t \Omega^t, \qquad O^0=I_n. \]
Then $E^t=(e^t_i)_{i=1}^n$ with $e^t_i:=\sum_{j=1}^n O^t_{ij} \bar{e}^t_j$
gives a canonical frame.
\end{proposition}

\begin{proof}
We will sometimes omit the evaluations at time $t$ for brevity.
First of all, $O^t$ is orthogonal for all $t$ since
\[ (O O^{\sT})'
 =\frac{1}{2} O \Omega O^{\sT} +\frac{1}{2} O \Omega^{\sT} O^{\sT}
 =\frac{1}{2} O (\Omega +\Omega^{\sT}) O^{\sT} =0, \]
where $O^{\sT}$ denotes the transpose of $O$.
To see that $E^t$ is orthonormal, observe
\[ \omega(\dot{E},E)
 =\frac{1}{2} \omega(O \Omega \ol{E},O \ol{E}) +\omega(O \dot{\ol{E}},O \ol{E})
 =\frac{1}{2} O \Omega \omega(\ol{E},\ol{E}) O^{\sT} +O \omega(\dot{\ol{E}},\ol{E}) O^{\sT}. \]
Then it follows from (the proof of) Lemma~\ref{lm:Lagsp}
and $\omega(\dot{\ol{E}},\ol{E}) \equiv I_n$ that
$\omega(\dot{E}^t,E^t) \equiv I_n$.

It remains to show $\ddot{e}^t_i \in \cJ^t_{\alpha}$.
We have
\[ \ddot{E} =\left( \frac{1}{2}O \Omega \ol{E} +O \dot{\ol{E}} \right)'
 =\frac{1}{2}(O \Omega)' \ol{E} +O \Omega \dot{\ol{E}} +O \ddot{\ol{E}}, \]
and hence
\[ \omega(\ddot{E},\ol{E})
 =\omega(O \Omega \dot{\ol{E}},\ol{E}) +\omega(O \ddot{\ol{E}},\ol{E})
 =O \Omega +O \omega(\ddot{\ol{E}},\ol{E})
 =O \left[ \omega(\dot{\ol{E}},\ol{E}) \right]' \equiv 0. \]
This implies $\ddot{e}^t_i \in \cJ^t_{\alpha}$
because the intersection of the kernels of $\omega(\cdot,\bar{e}^t_j)$,
$j=1,\ldots,n$, coincides with $\cJ^t_{\alpha}$.
$\qedd$
\end{proof}

\subsection{Coordinate representation}\label{ssc:coord}%%%%%%%%%%%%

We compute the curvature operator in coordinates by using the canonical frame in the previous subsection.
First of all, Proposition~\ref{pr:O^t} yields the following (see \cite[Proposition~3.3]{Le2}):
\begin{align}
\bR^0_{\alpha}(E^0)
&= -(O^t \ol{E}^t)''|_{t=0}
 =-\left( \frac{1}{2}O^t \Omega^t \ol{E}^t +O^t \dot{\ol{E}^t} \right)' \bigg|_{t=0}
 \nonumber\\
&= -\frac{1}{4}(\Omega^0)^2 \ol{E}^0 -\frac{1}{2}\dot{\Omega}^0 \ol{E}^0
 -\Omega^0 \dot{\ol{E}^0} -\ddot{\ol{E}^0} \label{eq:R^0}.
\end{align}

We choose a coordinate satisfying $\bH_{\alpha_i \alpha_j}(\alpha)=\delta_{ij}$ for all $i,j$,
so that we can take $\bm{\xi}_i(0)=\del_{\alpha_i}$.
Then it holds that, by the calculation in Lemma~\ref{lm:orth},
\begin{equation}\label{eq:e'}
\dot{\bar{e}}^0_i =-\del_{x^i} +\sum_{k=1}^n
 \{ \bH_{x^k \alpha_i}(\alpha)+\dot{\xi}_i^k(0) \} \del_{\alpha_k}.
\end{equation}
This implies
\begin{equation}\label{eq:Omega}
\Omega_{ij}^0 =\omega(\dot{\bar{e}}^0_i,\dot{\bar{e}}^0_j)
 =\bH_{x^i \alpha_j}(\alpha) -\bH_{x^j \alpha_i}(\alpha) +\dot{\xi}_j^i(0) -\dot{\xi}_i^j(0).
\end{equation}
Now, for further simplification,
let us assume $\bH_{\alpha_i \alpha_j}(\bm{\alpha}(t))=\delta_{ij}$
and put $\bm{\xi}_i(t)=\del_{\alpha_i}$ for all $t$.
It is always possible to choose such a coordinate by noticing that
\[ \left( \sum_{i=1}^n a_i \del_{x^i},\sum_{j=1}^n b_j \del_{x^j} \right)
 \ \longmapsto\ \sum_{i,j=1}^n a_i b_j \bL_{v^i v^j}\big( \dot{\eta}(t) \big),
 \quad \eta(t):=\pi_M \big( \bm{\alpha}(t) \big), \]
gives an inner product of $T_{\eta(t)}M$ for each $t$
and by taking $(x^i)_{i=1}^n$ such that $(\del_{x^i})_{i=1}^n$
is orthonormal with respect to this inner product at every $\eta(t)$.
Then we have, by omitting evaluations at $\alpha$ and $t=0$ for brevity,
\begin{align*}
\ddot{\bar{e}}_i^0
&= -\sum_{k=1}^n \left\{ [\bH_{\alpha_k \alpha_i} \circ \bm{\alpha}]' \del_{x^k}
  -[\bH_{x^k \alpha_i} \circ \bm{\alpha}]' \del _{\alpha_k} \right\} \\
&\quad +\sum_{k,l=1}^n \left\{
 \bH_{\alpha_k \alpha_i} (\bH_{\alpha_l x^k} \del_{x^l} -\bH_{x^l x^k} \del_{\alpha_l})
 -\bH_{x^k \alpha_i} (\bH_{\alpha_l \alpha_k} \del_{x^l} -\bH_{x^l \alpha_k} \del_{\alpha_l}) \right\} \\
&=\sum_{k=1}^n (\bH_{\alpha_k x^i}-\bH_{x^k \alpha_i}) \del_{x^k}
 +\sum_{k=1}^n \left\{ [\bH_{x^k \alpha_i} \circ \bm{\alpha}]' -\bH_{x^k x^i}
 +\sum_{l=1}^n \bH_{x^l \alpha_i} \bH_{x^k \alpha_l} \right\} \del_{\alpha_k}.
\end{align*}
Combining this with \eqref{eq:R^0}, \eqref{eq:e'} and \eqref{eq:Omega},
we see that the horizontal part of $\bR^0_{\alpha}(e^0_i)$ indeed vanishes
(so that $\bR^0_{\alpha}(e^0_i) \in \cV_{\alpha}$).
We finally calculate
\begin{align*}
\dot{\Omega}^0_{ij}
&=\omega(\dot{\bar{e}}^0_i,\ddot{\bar{e}}^0_j)
 +\omega(\ddot{\bar{e}}^0_i,\dot{\bar{e}}^0_j) \\
&= [\bH_{x^i \alpha_j} \circ \bm{\alpha}]'
 -\bH_{x^i x^j} +\sum_{k=1}^n \bH_{x^k \alpha_j} \bH_{x^i \alpha_k}
 +\sum_{k=1}^n \bH_{x^k \alpha_i}(\bH_{\alpha_k x^j} -\bH_{x^k \alpha_j}) \\ 
&\quad -[\bH_{x^j \alpha_i} \circ \bm{\alpha}]'
 +\bH_{x^j x^i} -\sum_{k=1}^n \bH_{x^k \alpha_i} \bH_{x^j \alpha_k}
 -\sum_{k=1}^n \bH_{x^k \alpha_j}(\bH_{\alpha_k x^i} -\bH_{x^k \alpha_i}) \\ 
&= [(\bH_{x^i \alpha_j} -\bH_{x^j \alpha_i}) \circ \bm{\alpha}]'.
\end{align*}
Consequently, we obtain that the coefficient of $\del_{\alpha_j}=e^0_j$ in
$\bR^0_{\alpha}(e^0_i)$ is
\begin{align}
&-\frac{1}{4} \sum_{k=1}^n
 (\bH_{x^i \alpha_k}-\bH_{x^k \alpha_i})(\bH_{x^k \alpha_j}-\bH_{x^j \alpha_k})
 -\frac{1}{2} [(\bH_{x^i \alpha_j} -\bH_{x^j \alpha_i}) \circ \bm{\alpha}]' \nonumber\\
& -\sum_{k=1}^n (\bH_{x^i \alpha_k}-\bH_{x^k \alpha_i}) \bH_{x^j \alpha_k}
 -[\bH_{x^j \alpha_i} \circ \bm{\alpha}]' +\bH_{x^j x^i}
 -\sum_{k=1}^n \bH_{x^k \alpha_i} \bH_{x^j \alpha_k} \nonumber\\
&= \frac{1}{4} \sum_{k=1}^n
 (\bH_{x^i \alpha_k}-\bH_{x^k \alpha_i})(\bH_{x^j \alpha_k}-\bH_{x^k \alpha_j})
 -\frac{1}{2} [(\bH_{x^i \alpha_j} +\bH_{x^j \alpha_i}) \circ \bm{\alpha}]' \nonumber\\
&\quad -\sum_{k=1}^n \bH_{x^i \alpha_k} \bH_{x^j \alpha_k} +\bH_{x^i x^j}.
\label{eq:Rij}
\end{align}
Note that this is indeed symmetric in $i$ and $j$.

{\small%%%%%%%%%%%%%%%%%%%%%%%%%%%%%%%%%%

}

\end{document}